\documentclass{article}
\pdfoutput=1
\usepackage{amsmath}
\usepackage{amsthm}
\usepackage{indentfirst}
\usepackage{subfigure}
\usepackage{graphicx,floatrow}
\usepackage{epsfig}
\usepackage[body={16cm,24cm}, top=3cm]{geometry}
\newtheorem{definition}{Definition}
\newtheorem{theorem}{Theorem}
\newtheorem{proposition}{Proposition}
\newtheorem{corollary}{Corollary}
\newtheorem{example}{Example}
\newtheorem{lemma}{Lemma}
\newtheorem{remark}{Remark}
\usepackage{algorithm, algorithmic}
\usepackage{amssymb}
\usepackage{indentfirst}
\author{Wang Changlong, Zhou Feng, Wang Weichao, Ren Kaiqiang}
\begin{document}
\title{The unified theoretical frame of logarithmic alternative model for sparse information recovery and its application in multiple source location problem by TDOA}
\date{}
\maketitle
\begin{abstract}
In sparse information recovery, the core problem is to solve the $l_0$-minimization which is NP-hard. On one hand, in order to recover the original sparse solution, there are a lot of papers designing alternative model for $l_0$-minimization. As one of the most popular choice, the logarithmic alternative model is widely used in many applications. In this paper, we present an unified theoretical analysis of this alternative model designed for $l_0$-minimization. By the theoretical analysis, we prove the equivalence relationship between the alternative model and the original $l_0$-minimization. Furthermore, the main contribution of this paper is to give an unified recovery condition and stable result of this model. By presenting the local optimal condition, this paper also designs an unified algorithm and presents the corresponding convergence result. Finally, we use this new algorithm to solve the multiple source location problem. 
\end{abstract}
\textbf{keywords:} sparse recovery, location, TD0A
\section{Introduction}
The aims of sparse recovery problem is to find the sparsest solution of an underdetermined equation $Ax=b$, where $A\in \mathbb{R}^{m\times n}$ with $m<n$, and these problems can be modelled as the following $l_0$-minimization,
\begin{eqnarray}
\min \limits_{x\in \mathbb{R}^n}\|x\|_0 \ \ s.t. \ Ax=b
\end{eqnarray}
where $\|x\|_0$ stands the number of non-zero elements of the vector $x$ and we call $x$ as a $k$-sparse vector when $\|x\|_0\leq k$. In recent years, sparse recovery has been applied widely in many application, such as visual coding \cite{olshausen1996emergence}, matrix completion \cite{candes2009exact}, source localization \cite{malioutov2005sparse}, and face recognition \cite{wright2008robust}. However, $l_0$-minimization has been proved to be a NP-Hard problem \cite{www}. Therefore, how to design algorithms to solve $l_0$-minimization has always been a lively discussed problem internationally in field of information theory. 

In order to solve $l_0$-minimization, the related international studies can be divided into two categories. One is to design algorithms to solve $l_0$-minimization directly, such as the OMP \cite{Liu2018The, Becerra2018A, Ji2018Fast, Lee2018Efficient}, Subspace pursuit algorithm \cite{dai2009subspace,temlyakov2011greedy} . Another type of algorithms is to design an alternative model for $l_0$-minimization and get the sparse solution by solving the alternative model. For example, inspired by LASSO model, the following $l_1$-minimization is used for getting the sparse solution \cite{Candes2005Decoding,Cand2008The,Foucart2009Sparsest,Zhang2017A},
\begin{eqnarray}
\min \limits_{x\in \mathbb{R}^n}\|x\|_1 \ \ s.t. \ Ax=b
\end{eqnarray}
In mathematics, $l_1$-minimization can be changed into a linear programming, so it can be solved effectively by many convex optimization algorithms \cite{Dantzig2003Linear,Todd1988Exploiting}. However, in order to ensure the solution of $l_1$-minimization is the sparse solution, the measurement matrix $A$ needs satisfies some strict conditions, such as RIP. A matrix $A$ is said to satisfy the RIP of order $k$ if the following inequality holds
\begin{eqnarray}
(1-\delta_k)\|x\|_2^2\leq \|Ax\|_2^2\leq (1+\delta_k)\|x\|_2^2
\end{eqnarray}
where $0<\delta_k<1$. Until now, there are a lot of conditions based on RIP to ensure the successful recovery by $l_1$-minimization, such as $\delta_{3k}+3\delta_{4k}<2$  in \cite{candes2006near}, $\delta_{2k}<{\sqrt{2}-1}$ in \cite{Cand2008The}, $\delta_{2k}<2(3-\sqrt{2})/7$ in \cite{Foucart2009Sparsest}. Cai and Zhang \cite{cai2013sparse} showed that for any given $t\geq \frac{4}{3}$, $\delta_{tk}<\sqrt{\frac{t-1}{t}}$ guarantees to recover every $k$ sparse vector by $l_1$-minimization.

\subsection{Greedy algorithms and alternative models}
Although greedy algorithms are designed to solve $l_0$-minimization directly, these algorithms do not perform well with a high spark level because the $l_0$-minimization model itself is an NP-HARD problem. For these greedy algorithm, the basic idea is to find the supporting index of the matching solution vector by comparing the similarity between the current residual and the measurement matrix column vectors, and then optimize the current supporting index by the least square method in each iteration. In principle, a $k$-sparse vectors can be recovered only by iterating $k$ steps. However, related theories show that OMP algorithm can guarantee successful recovery when the following inequality is satisfied
\begin{eqnarray}
k\leq \frac{1}{2}\left(1+\frac{1}{\kappa(A)}\right)
\end{eqnarray}
where $\kappa(A)=\max \limits_{i,j}\frac{|A_i^TA_j|}{\|A_i\|_2\|A_j\|_2}$. Although the coherence of the matrix is relatively easy to calculate, in fact, related theories show that $\kappa(A)\geq \sqrt{\frac{m}{n}}$ and how to obtain a sufficiently small coherence is an extremely difficult problem.
Except the prior of sparsity, greedy algorithms are greatly affected by the level of noise. In fact, the recovery condition of these algorithms are restrictive. For example, by the analysis of OMP, the algorithm will definitely fail as long as a support index update error occurs at a certain step \cite{www}.

Another type of solving $l_0$-minimization is to design an alternative function. In order to overcome the discrete nature of $0$-norm, the alternative function has property similar to $0$-norm and these functions are often more conducive to optimizing.
Except that we've already mentioned, there exist a lot of alternative models designed for $l_0$-minimization.

Except convex relaxation method, the following $l_p$-minimization is also a popular choice,
\begin{eqnarray}
\min \limits_{x\in \mathbb{R}^n}\|x\|_p^p \ \ s.t. \ Ax=b
\end{eqnarray}
where $\|x\|_p^p=\sum _{i=1}^n |x_i|^p$. Compared with $l_1$-minimization, it seems to be more natural to consider $l_p$-
minimization with $0<p<1$ instead of $l_0$-minimization than $l_1$-minimization. Fourcart \cite{www} showed 
that the condition $\delta_{2k}<0.4531$ can guarantee exact $k$-sparse recovery via $l_p$-minimization for any $0<p<1$. 
Chartrand \cite{chartrand2007exact} proved that, if $\delta_{2k+1}<1$, then we can recover a $k$-sparse vector by $l_p$-
minimization for some $p>0$ small enough. Consider that calculate $\delta_k$ of a given matrix $A$ is still a NP-HARD problem, 
Peng and Li \cite{peng2015np} give a more general conclusion about $l_p$-minimization. For any $A$ and $b$, there exists a 
constant $p(A,b)$ such that the solution of $l_p$-minimization is the original sparse solution. It is worth emphasizing that this conclusion is valid even the measurement matrix $A$ is not satisfied with RIP.

In recent years, there are many alternative models except $l_p$-minimization. To summarize, these alternative models can be expressed as the following $l_{f_p}$-minimization.
\begin{eqnarray}
\min \limits_{x\in \mathbb{R}^n}\|x\|_{f_p} \ \ s.t. \ Ax=b
\end{eqnarray}
where the function $f_p:\mathbb{R}\rightarrow \mathbb{R}^+$ and $\|x\|_{f_p}=\sum_{i=1}^nf_p(|x_i|)$. In Table 1, we present some popular design of $f_p(\cdot)$. 
\begin{table}\label{0124_0054}
\begin{center}
\caption{Some popular design of $f_p(\cdot)$}
\begin{tabular}{ccc} 
\hline
The alternative function & $p$ & Paper \\
\hline
$f_p(x)=|x|^p$ & $0<p\leq 1$ & \cite{peng2015np} \\
$f_p(x)=\frac{|x|}{|x|+p}$ & p>0 &\cite{Li2014Smooth,Cui2016Affine,Cui2018Sparse,zhang2017minimization} \\
$f_p(x)=log(1+|x|/p)$ & p>0 &\cite{Kiechle2015A,Han2016A} \\
$f_p(x)=1-e^{|x|/p}$ &p>0 &\cite{Le2015DC} \\
$f_p(x)=tanh(|x|/p)$ &p>0 & \cite{Chouzenoux2013A}\\
\hline
\end{tabular}
\end{center}
\end{table}
To summarize these non-convex alternative function $f_p(\cdot)$, the following properties are the source of design,
\begin{equation}
\begin{cases}
\forall p>0, \ f_p(0)=0 \\
\text{For any}\ x\neq 0, \ \lim \limits_{p\rightarrow 0}f_p(x)=1\\
\end{cases}
\end{equation}
It is easy to get that the function with the above condition is similar to $0$-norm when $p$ tends to $0$ and it is nature to admit these alternative function with a small $p$.
\subsection{Multiple source location problem and Sparse point represent method}
TIME difference of arrival (TDOA) measurements are
widely used in various applications of sensor networks,
e.g., source localization [1] and tracking [2]. In practice, they
can be obtained by generalized cross-correlation (GCC) [3]. By the classic theory of TDOA, three sensors are enough to locate a target on a plane area, however, precise localization of multiple source is a fundamental problem which has received an upsurge of attention recently [1]. As shown in Figure \ref{fig1-1}, there are two group sensors with different color and each group can only locate one target. Once one of the sensors do not work properly, the target $P_1$ or $P_2$ will miss. If we link these sensors into a wireless network, the network can locate two targets at the same time. What is more important, the network should work fine even one of sensors is miss. The method of sparse recovery have been widely used in multiple source location problems. Some classic sparse methods has been used for solving this problem, such as Greedy Matching Pursuit (GMP) algorithm , $l_1$-minimization and OMP. 

Sparse is a nature property of many physical processes. Furthermore, sparse recovery methods are widely used in Sparse Point represent problem. Usually, theses problem can be explained as the following.
There is a function $f:X\times Y\rightarrow \mathbb{C} \ or\  \mathbb{R}$ which stands a certain physical process , where $X$ is a $d$-dimension parameter space and $Y$ is the measurement space. Usually, the function $f$ satisfies the following property, 
\begin{eqnarray}
f\left(\bigcup\limits _{i=1}^kx^{(i)},y\right)=\sum\limits_{i=1}^kf(x^{(i)},y)=\varphi(y)
\end{eqnarray} 
for any $y\in Y$ and $x^{(i)}$. In such problems, the measurement $\varphi(y)$,$y\in \{y^{(1)},y^{(2)},...,y^{(m)}\}$ are known and the set $\{x^{(1)},x^{(2)},...,x^{(k)}\}$ is what we desire. Therefore, the inverse problem can be model as 
\begin{eqnarray}
D(\varphi(y^{(1)}),\varphi(y^{(2)})...,\varphi(y^{(m)}))\rightarrow \bigcup\limits _{i=1}^kx^{(i)}
\end{eqnarray} 
where the operator $D$ stands the process of solving this inverse problem. By dividing grid points in $X$, we can solve it by sparse recovery method. Let
\begin{equation}\label{0124_0128}
A=\left(
\begin{array}{cccc}
f(z^{(1)},y^{(1)}) & f(z^{(2)},y^{(1)}) & ... & f(z^{(n)},y^{(1)})  \\
f(z^{(1)},y^{(2)}) & f(z^{(2)},y^{(2)}) & ... & f(z^{(n)},y^{(2)})  \\
... & ... & ... & ... \\
f(z^{(1)},y^{(m)}) & f(z^{(2)},y^{(m)}) & ... & f(z^{(n)},y^{(m)})  \\
\end{array}
\right)
 b=\left(
\begin{array}{cc}
\varphi(y^{(1)})   \\
\varphi(y^{(2)}) \\
... \\
\varphi(y^{(m)})
\end{array}
\right)
\end{equation} 
Although the parameter $k$ is unknown, we can get an estimate by experience. We notice that the number of grid points $n$ can be any positive integer. Once $k<<n$, we can treat the $k$-sparse solution vector $x$ as a sparse vector in $\mathbb{R}^n$. Therefore, it is reasonable to solve the inverse by $l_0$-minimization where the measurement matrix $A$ and measurement vector $b$ are defined in (\ref{0124_0128}). Furthermore, in order to ensure $x^i$ can be represent by $z^i$, we can reduce grid distance and increase the number of grid points until meeting the demand of test precision.
\begin{center}
\begin{tabular}{ccccc}
\hline
Problems & $f$ & $X$& $Y$& Paper \\
\hline
ISAR image & $exp\left(-j\frac{4\pi}{c}(f_c+f_n)x\right)$& Scattering points & Pulse & \cite{su2013isar}\\
DOA estimation& $exp(-j2\pi0.5\sin x)$ & DOA & Sensors &\cite{zhang2019gridless}\\
Multiple source location & $\|x-x_1\|_2-\|x-x_n\|_2$ & grid points & Receivers & \cite{Jamali2013Sparsity} \\
Acoustic Scatter  & $exp(idk(x-\hat{x}))$ & location & Far-field data & \cite{wang2020new} \\
\hline
\end{tabular}
\end{center}
\begin{figure}[htb]
	\centering
	\subfigure[Classic location by TDOA]{\includegraphics[width=0.49\textwidth]{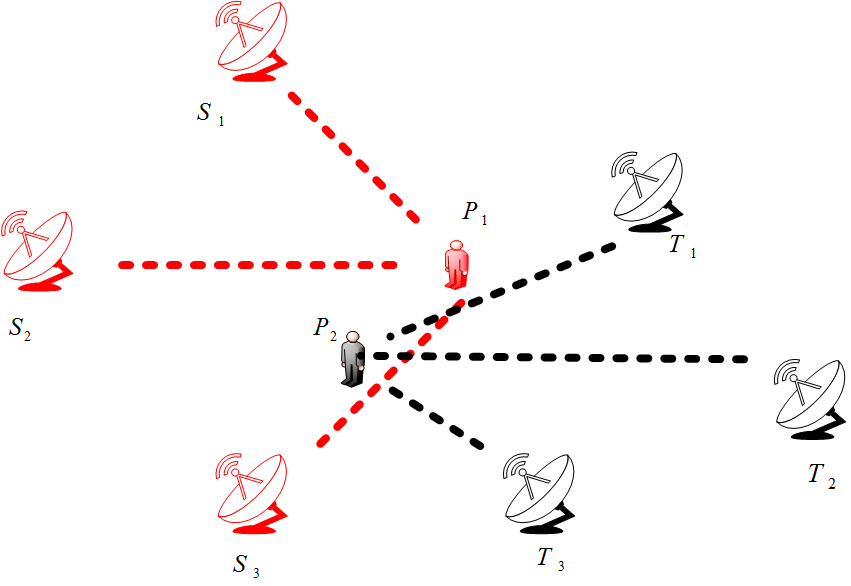}}
    \subfigure[Multiple source location]{\includegraphics[width=0.49\textwidth]{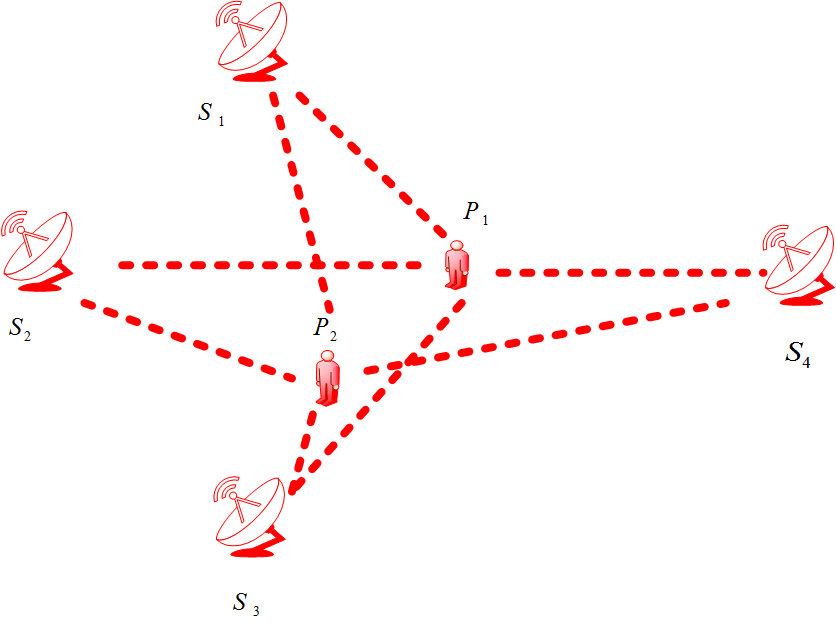}}
	\caption{The need of multi-source location.\label{fig1-1}}
\end{figure}
In Table 2, we show some problems which is solved by sparse point represent method and the details about the methods can be found in the corresponding references. For sparse point methods, we have to increase the number of grid points to meet the resolution require. However, reduce the distance between grid points can make the coherence of measurement matrix $A$ increase, so that algorithms may fail even the sparse level is low. On the other hand, the measurement matrix in sparse point methods is difficult to satisfy with RIP since these matrices do not have random structure. 
\subsection{The main contribution of this paper}
In this paper, we consider the alternative function $h_{p,q}(x)=log(1+\frac{|x|^q}{p})$ instead of the original $0$-norm function and give some theory analysis of this alternative function. Furthermore, we use this new model to solve multiple source location problem by TDOA. To summarize, the main contribution of this paper is the following.

\romannumeral1. Both noise and noiseless case, we prove the equivalence relationship between $l_{h_{p,q}}$-minimization and $l_0$-minimization.

\romannumeral2. We present the recovery condition and the stable result of $l_{h_{p,q}}$-minimization.

In this paper, a sufficient and necessary condition of $l_{h_{p,q}}$-minimization is given and we also show that what matrices satisfy such condition. By a new concept named, we also show the stable result of $l_{h_{p,q}}$-minimization.

\romannumeral3. By presenting an analysis expression of local optimal solution, an unify algorithm for $l_{h_{p,q}}$-minimization and the corresponding convergence conclusion are given.

\romannumeral4. Multiple source location problem can be modelled as a sparse recovery model. By applying $l_{h_{p,q}}$-minimization, we solve it by our algorithm, and experiment result show that our method do better than classic algorithm in sparse recovery.

Our paper is organized as follows. In Section 2, the equivalence relationship between $l_{h_{p,q}}$-minimization and $l_0$-minimization is established. Both noise case and noiseless case, the sparse solution can be recovered by minimizing $h_{p,q}(\cdot)$. In Section 3, some theoretical analysis are presented, including the recovery condition and stable result of  $l_{h_{p,q}}$-minimization. By giving the analytic expression of local optimizations, we give a fixed point iterative algorithm and the corresponding convergence conclusion of $l_{h_{p,q}}$-minimization in Section 4. Finally, we apply $l_{h_{p,q}}$-minimization to solve multiple source location problem.
\subsection{Some symbols}
For convenience, for $x \in \mathbb{R}^n$, its support is defined by $support\ (x)=\{i:x_i \neq 0\}$ and the cardinality of set $\Omega$ is denoted by $|\Omega|$.
Let $Ker(A)=\{x \in \mathbb{R}^n:Ax=0\}$ be the null space of matrix $A$. We define subscript notation $x_\Omega$ to be such a vector that is equal to $x$ on the index set $\Omega$ and zero everywhere else. and use the subscript notation $A_\Omega$ to denote a submatrix whose columns are those of the columns of $A$ that are in the set index $\Omega$. Let $\Omega^c$ be the complement of $\Omega$. 
For any positive integer $n$, we denote $[n]=\{1,2,3,...n\}$.
\section{The equivalence relationship between the logarithmic alternative model and $l_0$-minimization}
As we have introduced in Section 1, the main work of this paper is to consider the alternative function $h_{p,q}(x)=\log(1+|x|^q/p)$ with $p>0$ and $0<q\leq 1$. i.e., we can get the following $l_{h_{p,q}}$-minimization,
\begin{eqnarray}\label{hp-model}
\min\limits_{x\in \mathbb{R}^n} \|x\|_{h_{p,q}} s.t. Ax=b
\end{eqnarray}
where $\|x\|_{h_{p,q}}=\sum_{i=1}^nh_{p,q}(x_i)$. For any $p>0$ and $0<q\leq 1$, it is easy to get that $h_{p,q}(0)=0$ and $\lim_{p\rightarrow 0^+}\frac{h_{p,q}(x)}{log(1+p^{-1})}=1$. Therefore, the main purpose of this section is to prove the equivalence between $l_{h_{p,q}}$-minimization and $l_0$-minimization.
In many papers which design alternative function, they only consider the condition (), and it is easy to get that the function $h_{p,q}(x)$ is also satisfied with condition () for any fixed $0<q\leq 1$. In two dimension cases, the graph of $h_{p,q}(x)$ with different $q$ are shown in Figure, it is obvious that $h_{p,q}(x)$ tends to $0$-norm as long as $p\rightarrow 0$. So it seems to be reasonable to believe that the sparse solution can be recovered by minimizing such alternative function with a small enough $p$. However, the following example will show us the condition () is not enough to ensure the equivalence relationship between alternative models and original problem.
\begin{figure}[!ht]
	\centering
	\subfigure[$q=1$]{\includegraphics[width=0.49\textwidth]{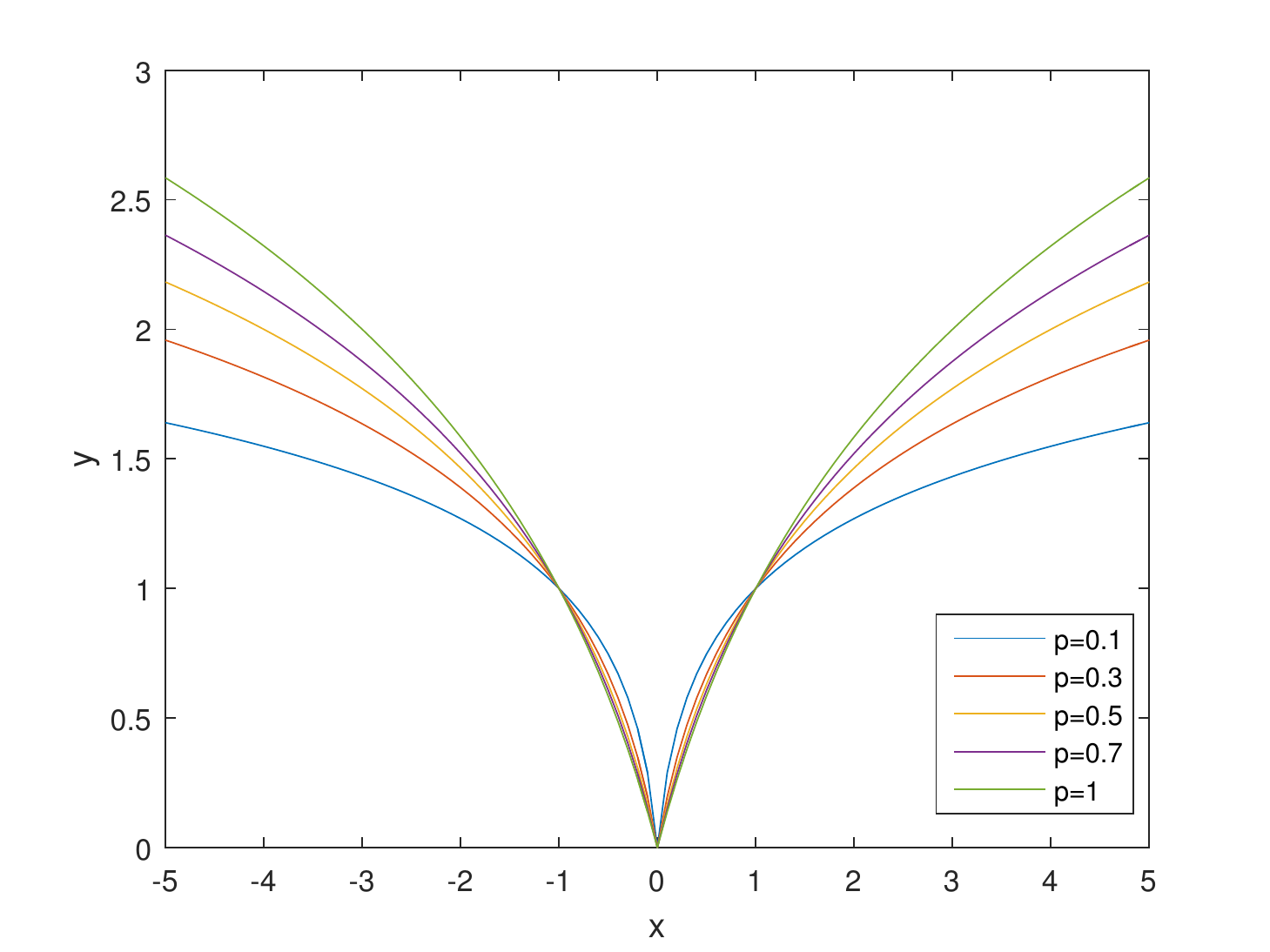}} \hfill
	\subfigure[$q=0.7$]{\includegraphics[width=0.49\textwidth]{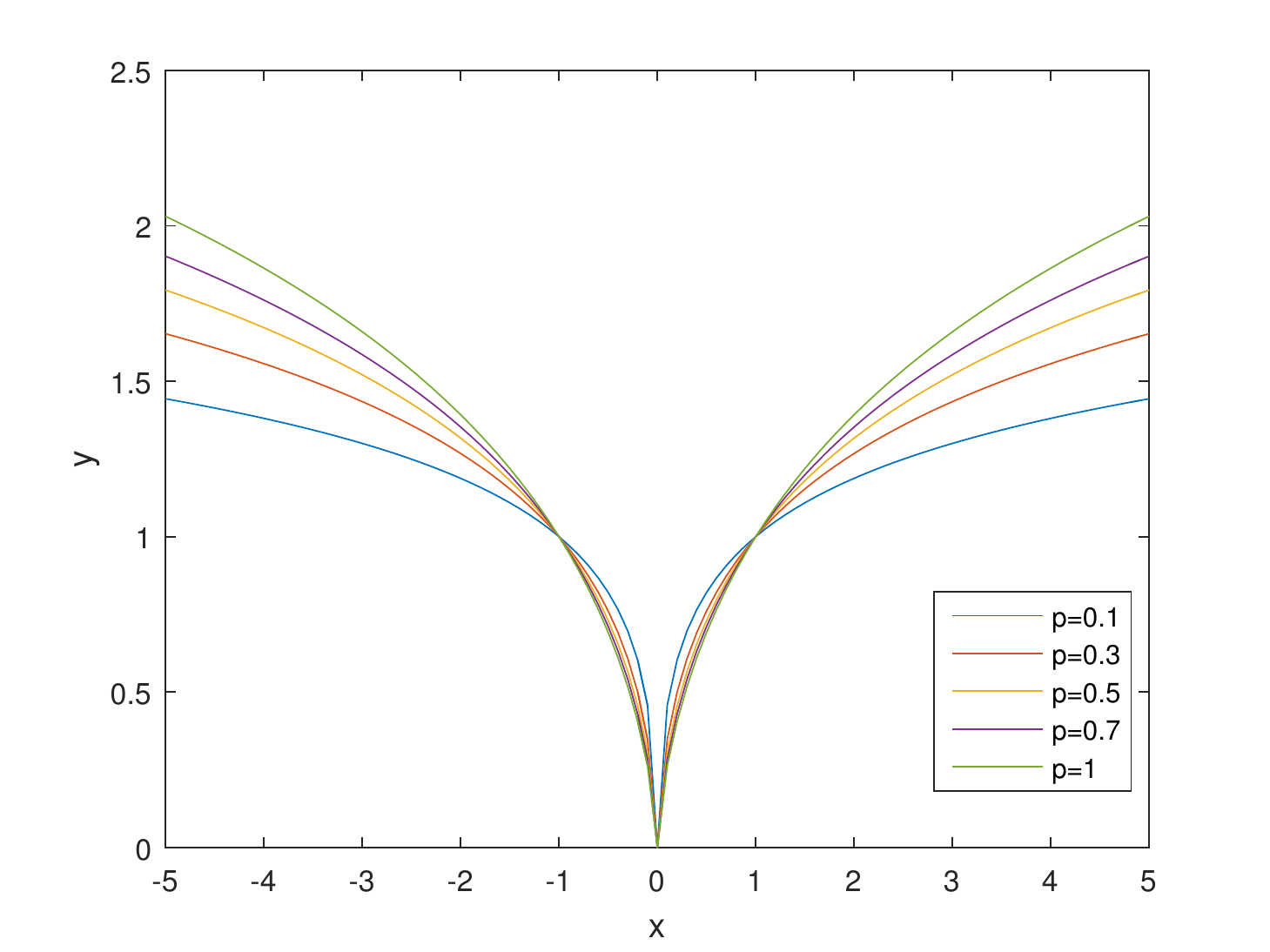}}
	\subfigure[$q=0.5$]{\includegraphics[width=0.49\textwidth]{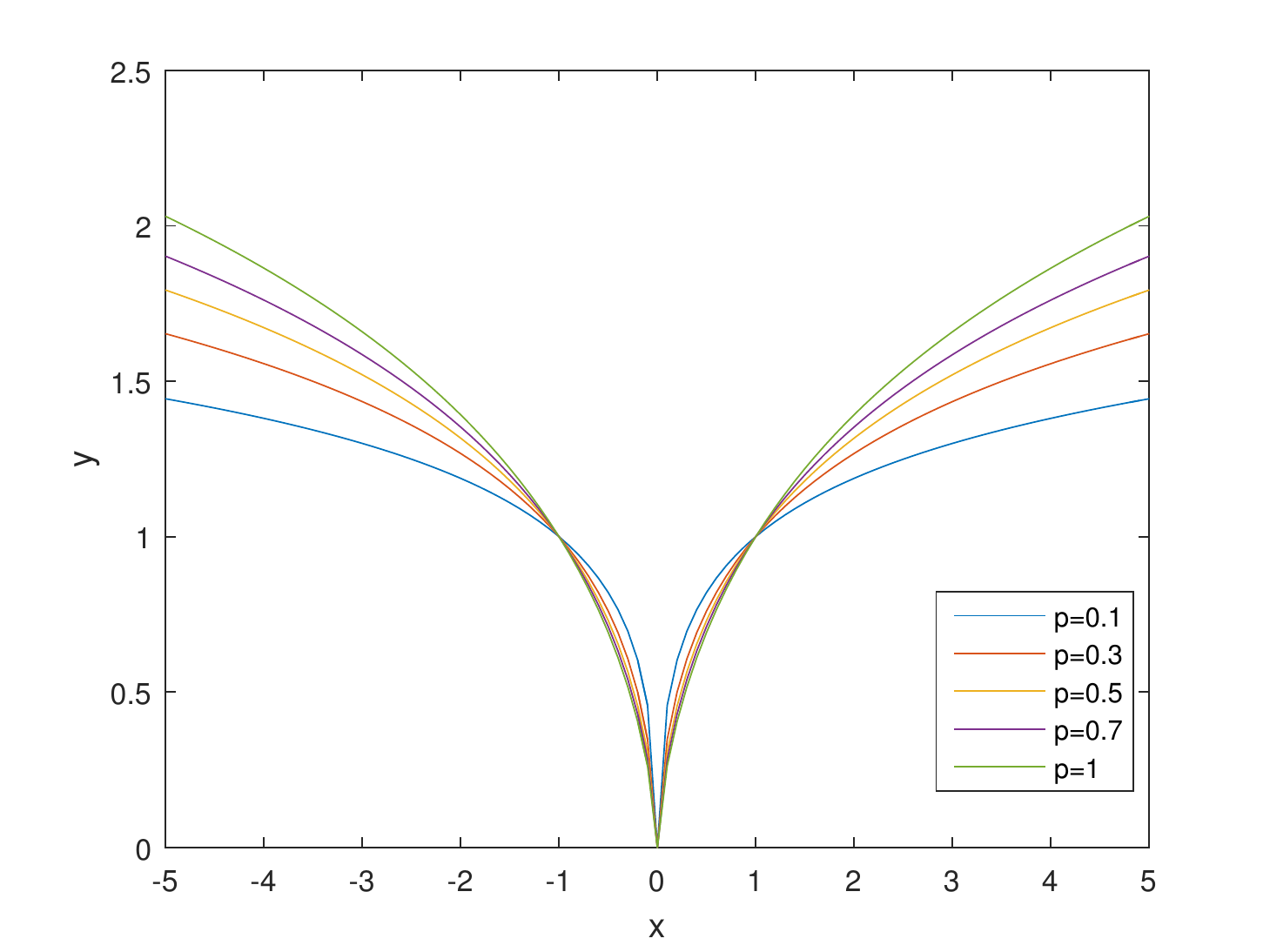}}
	\subfigure[$q=0.3$]{\includegraphics[width=0.49\textwidth]{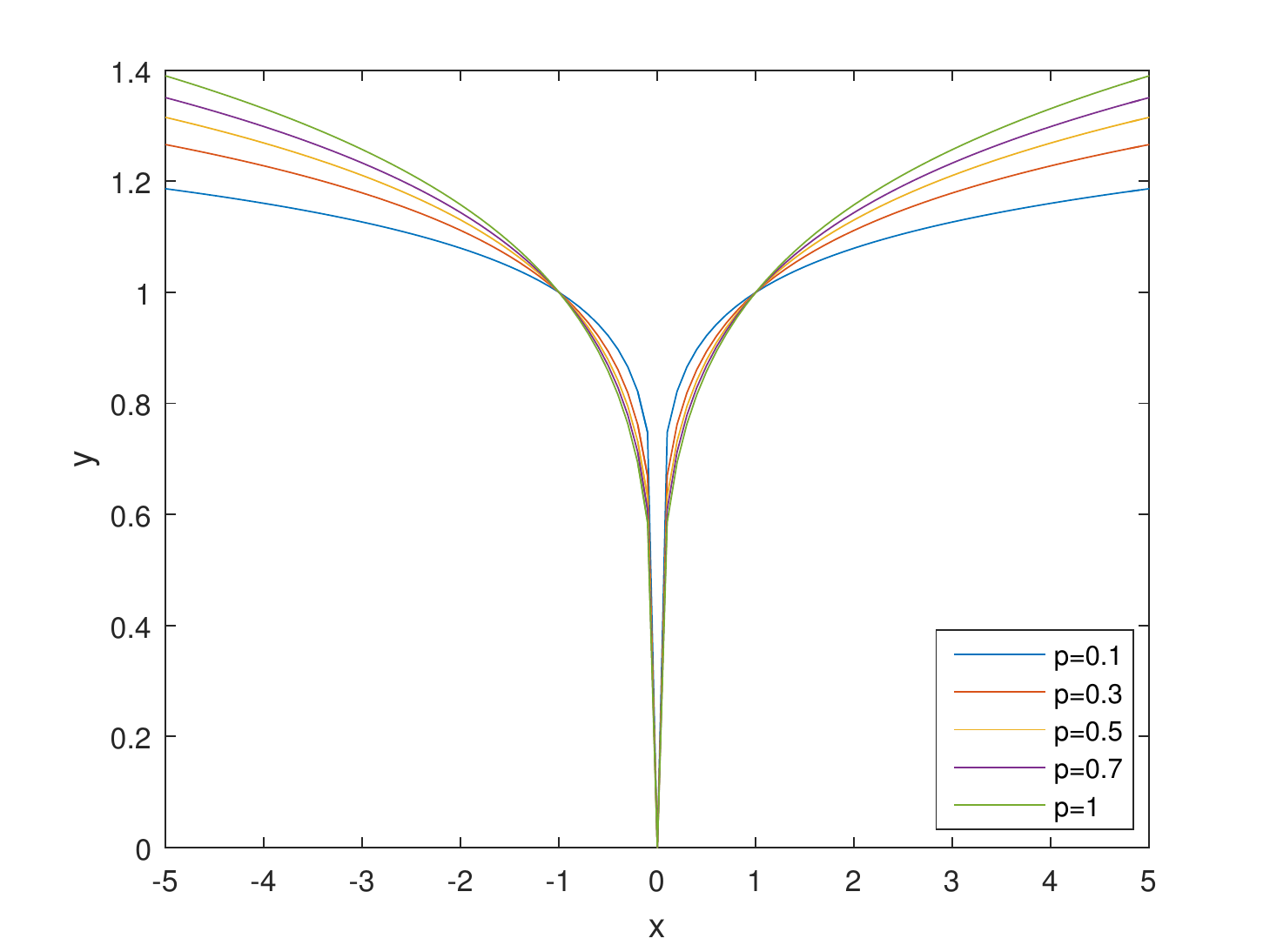}}
	\caption{With different values of $q$, the figure of $h_{p,q}(\cdot)$.\label{fig1-2}}
\end{figure}
\begin{example}
Consider the following measurement matrix $A$ and the measurement vector $b$,
\begin{center}
$A=\left(
\begin{array}{ccc}
1 & 0 & -1  \\
0 & 1 & -1  \\
\end{array}
\right)$ and $\boldsymbol b=\left(
\begin{array}{cc}
1   \\
0   
\end{array}
\right)$
\end{center}

It is easy to get the sparse solution of $Ax=b$, $x^*=(1,0,0)^T$ and the solution can be expressed as $x(t)=x^*+t(1,1,1)^T$, where $t\in \mathbb R$.

Consider the alternative function
\begin{eqnarray}
f_p(x)=
\begin{cases}
\frac{x}{p}\sin \frac{p}{x} & x\neq 0 \\
0& x=0
\end{cases}
\end{eqnarray}
It is obvious that $f_p(0)=0$ for any $p>0$,and $\lim \limits_{p\rightarrow 0} f_p(x)=1$, for any $x\neq 0$.

For a fixed $p\in(0,\pi/2)$, define the function $F_p(t)$ as
\begin{eqnarray}
F_p(t)=\|x(t)\|_{f_p}=f_p(1+t)+2f_p(t)
\end{eqnarray}
Consider a sequence $\{t_n=(p(1.5\pi+2n\pi))^{-1}\}\rightarrow 0$, such that
\begin{eqnarray}
F_p(t_n)=f_p(t_n)-\frac{2t_n}{p}
\end{eqnarray}

Consider the function $g(t)$
\begin{eqnarray}
g(t)=\frac{1+t}{p}\sin \frac{p}{1+t}-\frac{2t}{p}
\end{eqnarray}
therefore,
\begin{eqnarray}
g'(t)=\frac{1}{p}\sin \frac{p}{1+t}-\frac{1}{1+t}\cos \frac{p}{1+t}-\frac{2}{p}
\end{eqnarray}
and
\begin{eqnarray}
g'(0)=\frac{1}{p}\sin p-\cos p-\frac{2}{p}
\end{eqnarray}

Consider $h(p)$
\begin{eqnarray}
h(p)=\sin p-p\cos p-2
\end{eqnarray}
since $h'(p)=p\sin p>0$, we can get that $-1>h(p)>h(0)=-2$, when $p\in (0,\pi/2)$. So $g'(0)<0$.

Since $g'(x)$ is continuous at point $x=0$, there exists $(0,\xi)$ such that$g(t)<g(0)$ for any $t\in (0,\xi)$. For the sequence $\{t_n\}$, there exists an interger $N$ such that if $n>N$
\begin{eqnarray}
F_p(t_n)=g(t_n)<g(0)=F_p(0)
\end{eqnarray}
\end{example}
 
By Example 1, the condition () can not provide design principles for reasonable alternative functions. So it is worth considering whether the $h_{p,q}(x)$ is a correct choice. Further, we will prove the equivalence between alternative model and $l_0$-minimization. i.e., the solution of $l_{h_{p,q}}$-minimization is also the solution of original model. Before giving the main theorem, some lemmas are need to be presented.
\begin{lemma}\label{0731_1707}
If $x^*$ is the solution of the $l_{h_{p,q}}$-minimization (\ref{hp-model}), then the sub-matrix $A_S$ is a full rank matrix, where $S=support(x^*)$.
\end{lemma}
\begin{proof}
If the sub-matrix $A_S$ is not full rank, then there exists a vector $h \in Ker(A)$ such that $support(h)\subseteq S$. Without of generality, we note that$x^*=[x^*_1,x^*_2,...,x^*_n]^T$. Let
\begin{equation}
a=\max \limits_{h_i\neq 0} \frac{|x^*_i|}{|h_i|}
\end{equation}
therefore, it is easy to proof that 
\begin{equation}
\notag sign(x^*_i+\alpha h_i)=sign(x^*_i-\alpha h_i)=sign(x^*_i)
\end{equation}
for $\alpha \in [0,a]$. Since $h_{p,q}(x)$ is a concave function when $x\geq 0$ or $x\leq 0$, it is easy to get that
\begin{eqnarray}
\notag h_{p,q}(|x^*_i|)&=&h_{p,q}(\frac{1}{2}|x^*_i+\alpha h_i|+\frac{1}{2}|x^*_i-\alpha h_i|)\\
\notag &\geq & \frac{1}{2}h_{p,q}(|x^*_i+\alpha h_i|)+\frac{1}{2}h_{p,q}(|x^*_i-\alpha h_i|)
\end{eqnarray}
Therefore,
\begin{equation}
\notag h_{p,q}( x^*)\geq \frac{1}{2}h_{p,q}(x^*+\alpha h )+\frac{1}{2}h_{p,q}(x^*-\alpha h )
\end{equation}
Since $x^*$ is the solution of the $l_{h_{p,q}}$-minimization (\ref{hp-model}), we can get that, 
\begin{equation}
h_{p,q}(x^*_i) = \frac{1}{2}h_{p,q}(x^*_i+\alpha h_i )+\frac{1}{2}h_{p,q}(x^*_i+\alpha h_i)
\end{equation}
for any $i\in S$, which is obviously false.
\end{proof}
By lemma \ref{0731_1707}, we define the following set,
\begin{eqnarray}
T(A,b)=\{x\in \mathbb{R}^n| \ Ax=b \ and A_{support(x)}\ is\ full\ rank. \}
\end{eqnarray}
It is obvious that the set $T(A,b)$ is a finite set so we can define the following concept,
\begin{eqnarray}
r(A,b)=\max \limits_{x\in T(A,b)} |x_i|
\end{eqnarray}

In convex analysis, the recession cone of an unbounded convex set is an important concept. 
\begin{definition}
For a given convex set $T$, its linear recession cone $0^+T$ is defined as the following,
\begin{eqnarray}
\notag 0^+T&=&\{x|\forall \lambda>0, \ \forall y\in T, y+\lambda x\in T\}\\
           &=&\{x|x=\lim\limits_{\lambda_n\rightarrow 0^+}\lambda_nx_n,\ where \ x_n\in T\}
\end{eqnarray}
\end{definition}
For an underdetermined equations, $l_{h_{p,q}}$-minimization may admit several solutions so we need the following concept.
\begin{definition}
The nonincreasing rearrangement of the vector $x\in \mathbb{R}^n$ is the vector $\mu (x)\in \mathbb{R}^n$ for which
\begin{eqnarray}
\mu (x)_1\geq \mu (x)_2 \geq \mu (x)_3...\geq \mu (x)_n\geq 0
\end{eqnarray}
and there is a permutation $\pi$:$[n]\rightarrow [n]$ with $\mu (x)_{j}=|x_{\pi (j)}|$ for all $j\in [n]$. 
\end{definition}
The following lemma describes the relationship between a polynomial and its roots and will be a useful tool for the main theorem.
\begin{lemma}\label{1219_2121}
For given $w_1,w_2,...,w_n$, if the following equations 
\begin{eqnarray}\label{0814_2214}
\left \{
\begin{array}{c}
z_1+z_2...+z_n=w_1\\
z_1^2+z_2^2...+z_n^2=w_n^2 \\
.... \\
z_1^n+z_2^n...+z_n^n=w_n^n
\end{array}
\right.
\end{eqnarray} 
have a solution vector $z=[z_1,z_2,...,z_n]^T$, then the $\mu (z)$ is unique.
\end{lemma}
\begin{proof}
We denote that
\begin{eqnarray}
T_k=\sum_{\substack{i_1,i_2,...i_k\in [n]\\ i_a \neq i_b \\ a,b \in [k]}} z_{i_1}z_{i_2}...z_{i_k}
\end{eqnarray}
and
\begin{eqnarray}
U_k=\sum \limits_{i=1}^n z_i^k
\end{eqnarray}
By the binomial theorem, we have that
\begin{eqnarray}
\prod _{i=1}^n(z-z_i)=z^n+\sum_{k=1}^n T_kz^{n-i}
\end{eqnarray}
Therefore, in order to prove this lemma, it is enough to prove that $T_k$ can be represented uniquely by $U_1,...U_n$. We use mathematical induction to prove this conclusion. Notice that
\begin{eqnarray}
T_2=(U_1)^2-U_2
\end{eqnarray}
Assume that there exists a function $f_k(\cdot)$ such that
\begin{eqnarray}
T_k=f_k(U_1,U_2,...U_k)
\end{eqnarray}
It is easy to get that
\begin{eqnarray}
\notag T_k\cdot U_1&=&(k+1)T_{k+1}+\sum_{\substack{i_1,i_2,...i_k\in [n]\\ i_a \neq i_b \\ a,b \in [k]}} z_{i_1}^2z_{i_2}...z_{i_k} \\
\notag &=& (k+1)T_{k+1}+\sum\limits_{i=1}^n z_i^2 \sum_{\substack{i_1,i_2...i_{k-1}\in [n]\setminus i \\ i_a \neq i_b \\ a,b\in [k-1]}} z_{i_1}z_{i_2}...z_{i_{k-1}} \\
\notag &=& (k+1)T_{k+1}+\sum\limits_{i=1}^n z_i^2 \left(T_{k-1}-z_iT_{k-2}\right) \\
&=& (k+1)T_{k+1}+T_{k-1}U_2-T_{k-2}U_3
\end{eqnarray}
Therefore, we can conclude that
\begin{eqnarray}
T_{k+1}=\frac{1}{k+1}\left(T_{k-1}U_2-T_{k-2}U_3-T_kU_1\right)
\end{eqnarray}
The proof is complicated.
\end{proof}
\begin{lemma}\label{1219_2132}
For fixed vectors $x$,$y\in \mathbb{R}^n$ and $0<q\leq 1$, if there exist $(0,p^*]$ such that
\begin{eqnarray}
\|x\|_{h_{p,q}}=\|y\|_{h_{p,q}}
\end{eqnarray} 
whenever $p\in (0,p^*]$, then $\mu (x)=\mu (y).$
\end{lemma}
\begin{proof}
For such $x$,$y\in \mathbb{R}^n$ and $q$, we consider the following function
\begin{eqnarray}
g(p)=\|x\|_{h_{p,q}}-\|y\|_{h_{p,q}}
\end{eqnarray}
Without of generality, we assume $x_i$,$y_i\geq 0$. It is obvious that $g(p)$ is a smooth function when $p\in (0,p^*]$. Since $\|x\|_{h_{p,q}}=\|y\|_{h_{p,q}}$, it is easy to get that 
\begin{eqnarray}
g'(p)=\sum \limits_{i=1}^n\frac{-x_i^q}{p^2+px_i^q}-\sum \limits_{i=1}^n\frac{-y_i^q}{p^2+py_i^q}=0
\end{eqnarray}
Therefore,
\begin{eqnarray}\label{0814_2144}
\sum \limits_{i=1}^n\frac{1}{p+x_i^q}=\sum \limits_{i=1}^n\frac{1}{p+y_i^q}
\end{eqnarray}
Take derivative with respect to $p$ on both sides of the equation (\ref{0814_2144}),
\begin{eqnarray}
\sum \limits_{i=1}^n\frac{1}{(p+x_i^q)^2}=\sum \limits_{i=1}^n\frac{1}{(p+y_i^q)^2}
\end{eqnarray}
Continuous, it is easy to get that
\begin{eqnarray}\label{0814_2205}
\sum \limits_{i=1}^n\frac{1}{(p+x_i^q)^j}=\sum \limits_{i=1}^n\frac{1}{(p+y_i^q)^j}
\end{eqnarray}
for any positive integers $j\in N^+$.

Consider solving the following equations
\begin{eqnarray}
\left \{
\begin{array}{c}
z_1+z_2...+z_n=w_1\\
z_1^2+z_2^2...+z_n^2=w_n^2 \\
.... \\
z_1^n+z_2^n...+z_n^n=w_n^n
\end{array}
\right.
\end{eqnarray} 
where $w_j=\sum \limits_{i=1}^n\frac{1}{(p+y_i^q)^j}$. By lemma \ref{1219_2121}, it is obvious that both $x=(x_1,x_2,...x_n)^T$ and $y=(y_1,y_2,...y_n)^T$ are solution of equations (\ref{0814_2214}), so the vectors $x$ and $y$ share the same components values, i.e.,
\begin{eqnarray}
\mu (x)=\mu (y).
\end{eqnarray}
\end{proof}
\begin{lemma}\label{1219_2320}
For any fixed $x,y\in \mathbb{R}^n$ and $0<q\leq 1$, if $\|x\|_0=\|y\|_0$ and $\mu(x)\neq \mu (y)$, then there exist a constant $p^*_q(x,y)$ based on $x$ and $y$ such that
\begin{eqnarray}
\left(\|x\|_{h_{\alpha,q}}-\|y\|_{h_{\alpha,q}}\right)\left(\|x\|_{h_{\beta,q}}-\|y\|_{h_{\beta,q}}\right)> 0
\end{eqnarray}
for any $\alpha,\beta \in (0,p*(x,y))$ 
\end{lemma}
\begin{proof}
By lemma \ref{1219_2132}, we can conclude that there exists no $p^*$ such that $\|x\|_{h_{p}}=\|y\|_{h_{p}}$ for any $p\in (0,p^*)$. So there exist two sequences $\{p_n\},\{q_n\}\longrightarrow 0$ such that
 \begin{eqnarray}
\left(\|x\|_{h_{p_n,q}}-\|y\|_{h_{p_n,q}}\right)\left(\|x\|_{h_{q_n,q}}-\|y\|_{h_{q_n,q}}\right)< 0
\end{eqnarray}
Since $h_{p,q}(\cdot)$ is a continuous function, we can get a new sequence $\{\hat{p}_n\}$ such that
\begin{eqnarray}\label{1107_1155}
\|x\|_{h_{\hat{p}_n,q}}=\|y\|_{h_{\hat{p}_n,q}}
\end{eqnarray}
and $\hat{p}_n\in (min\{p_n,q_n\},max\{p_n,q_n\})$.
Rewrite (\ref{1107_1155}),
\begin{eqnarray}
\sum \limits_{i=1}^nlog(1+\hat{p}_n^{-1}|x_i|^q)=\sum \limits_{i=1}^nlog(1+\hat{p}_n^{-1}|y_i|^q)
\end{eqnarray}
Define a new function
\begin{eqnarray}
\varphi_1(p)=\sum \limits_{i=1}^nlog(1+\hat{p}_n^{-1}|x_i|^q)-\sum \limits_{i=1}^nlog(1+\hat{p}_n^{-1}|y_i|^q)
\end{eqnarray}
Since $\varphi_1 (\hat{p}_n)=0$, we can get a new sequence $\{\hat{p}_n^{1}\}\rightarrow 0$ such that $\varphi_1 '(\hat{p}_n^{(1)})=0$, i.e.,
\begin{eqnarray}
\sum \limits_{i=1}^n\frac{|x_i|^q}{|x_i|^q+\hat{p}_n^{(1)}}=\sum \limits_{i=1}^n\frac{|y_i|^q}{|y_i|^q+\hat{p}_n^{(1)}}
\end{eqnarray}
By the above inequality, we can get that 
\begin{eqnarray}
\sum \limits_{i=1}^n \frac{1}{\hat{p}_n^{(1)}}\left(1-\frac{|x_i|^q}{|x_i|^q+\hat{p}_n^{(1)}}\right)=\sum \limits_{i=1}^n \frac{1}{\hat{p}_n^{(1)}}\left(1-\frac{|y_i|^q}{|y_i|^q+\hat{p}_n^{(1)}}\right)
\end{eqnarray}
Since $\{\hat{p}_n^{1}\}\rightarrow 0$, we can get that
\begin{eqnarray}
\sum \limits_{i=1}^n\frac{1}{|x_i|}=\sum \limits_{i=1}^n\frac{1}{|y_i|}.
\end{eqnarray}
Define a new function
\begin{eqnarray}
\varphi _1(p)=\sum \limits_{i=1}^n\frac{1}{p+|x_i|}=\sum \limits_{i=1}^n\frac{1}{p+|y_i|}
\end{eqnarray}
Since  $\varphi_1(\hat{p}_n^{1})=0$, we can get a new sequence $\{\hat{p}_n^{(2)}\}\rightarrow 0$ such that $\varphi _2'(\hat{p}_n^{(2)})=0$.
It is easy to get that 
\begin{eqnarray}
\sum \limits_{i=1}^n\frac{1}{|x_i|^2}=\sum \limits_{i=1}^n\frac{1}{|y_i|^2}.
\end{eqnarray}

Repeat these action, we can get that 
\begin{eqnarray}
\sum \limits_{i=1}^n\frac{1}{|x_i|^j}=\sum \limits_{i=1}^n\frac{1}{|y_i|^j}.
\end{eqnarray}
for $j=0,1,2....n-1$.

By lemma \ref{1219_2121}, we can conclude that $x$ and $y$ share the same elements which contradicts $\mu(x)=\mu(y)$.
\end{proof}
\begin{lemma}
If $f(x)$ is a proper convex function which satisfies the following condition,
\begin{eqnarray}\label{0731_1601}
\varliminf \limits_{\lambda \rightarrow +\infty} f(x+\lambda y)< +\infty
\end{eqnarray}
for the fixed $x$ and $y$, then the function $f(x+\lambda y)$ is a decreasing function for $\lambda$.
\end{lemma}
\begin{proof}
For the fixed $x$ and $y$, let $h(\lambda)=f(x+\lambda y)$. Then we will prove that $h(\lambda)$ is a convex function.

For any $\lambda _1$, $\lambda _2$ and $\lambda \in [0,1]$, it is easy to get that 
\begin{align}
h(\lambda \lambda _1+(1-\lambda)\lambda _2)&=f(x+(\lambda \lambda _1+(1-\lambda)\lambda _2)y) \nonumber \\
&=f(\lambda (x+\lambda _1y)+(1-\lambda)(x+\lambda _2y)) \nonumber \\
&\leq \lambda f(x+\lambda _1y)+(1-\lambda) f_(x+\lambda _2y) \nonumber \\
&=\lambda h(\lambda _1)+(1-\lambda)h(\lambda _2)
\end{align}
therefore, $h(\lambda)$ is a convex function.

By the condition (\ref{0731_1601}), it is easy to get that 
\begin{eqnarray}
\varliminf \limits_{\lambda \rightarrow +\infty} f(x+\lambda y)=a< +\infty
\end{eqnarray}
Therefore, there exist a sequence $(\lambda _1,a)$,$(\lambda _2,a)$...$(\lambda _n,a)$ in the set $epi \ h(\lambda)$, where $\lim \limits_{n\rightarrow +\infty}\lambda _n=+\infty$. For any $\lambda _i$ and $\lambda _j$, it is easy to get that 
\begin{eqnarray}
h(\lambda \lambda _i+(1-\lambda)\lambda _j)\leq \lambda h(\lambda _i)+(1-\lambda)h(\lambda _j)
\end{eqnarray}
So we can conclude that the set $(1,0)\in 0^+(epi h(\lambda))$ which means that 
\begin{eqnarray}
h(\lambda+\lambda _i)\leq h(\lambda)
\end{eqnarray}
for any $\lambda$ and $\lambda _i>0$
\end{proof}
\begin{theorem}\label{0123_0040}
For any given $A\in \mathbb{R}^{m\times n}$,$b\in \mathbb{R}^m$ and $0<q\leq 1$, there exists a constant $p^*(A,b,q)$ such that the solution of $l_{h_{p,q}}$-minimization is the solution $l_0$-minimization whenever $0<p<p^*(A,b,q)$.
\end{theorem}
\begin{proof}
By Lemma \ref{0731_1707}, both the solutions of $l_0$-minimization and $l_{h_{p,q}}$-minimization are linear representation of  linearly independent column vectors. It is easy to get that $l_0$-minimization is equivalent to the following optimal problem,
\begin{eqnarray}\label{1219_2301}
\notag & \min\limits_{x\in \mathbb{R}^n} \|x\|_0\\
&s.t.\quad
\begin{cases}
Ax=b \\
\|x\|_{\infty}\leq r(A,b).
\end{cases}
\end{eqnarray}
and $l_{h_{p,q}}$-minimization is equivalent to the following optimal problem,
\begin{eqnarray}\label{1219_2310}
\notag & \min\limits_{x\in \mathbb{R}^n} \|x\|_{h_{p,q}}\\
&s.t.\quad
\begin{cases}
Ax=b \\
\|x\|_{\infty}\leq r(A,b).
\end{cases}
\end{eqnarray}
where $r(A,b)$ is defined in. Let $ e_{n}\in \mathbb{R}^{n}$ be a vector whose elements are all one, and $sign(\boldsymbol x)=[sign(x_1),sign(x_2),...,sign(x_n)]^T$. So we can rewrite the model (\ref{1219_2301}) as
\begin{equation}
\begin{split}
&\min\,\,  < e_{n},sign(z)> \\
& s.t.\quad
\begin{cases}
A x= b\\
- z\leq  x \leq  z\\
-r(A,b) e_n\leq  x \leq r(A,b) e_n
\end{cases}
\end{split}
\end{equation}
We also can rewrite model (\ref{1219_2310}) as
\begin{equation}
\begin{split}
&\min\,\,  \sum \limits_{i=1}^n h_{p,q}(z_i) \\
& s.t.\quad
\begin{cases}
Ax=b\\
- z\leq  x \leq  z\\
-r(A,b) e_n\leq  x \leq r(A,b) e_n
\end{cases}
\end{split}
\end{equation}
It is easy to find that the above two models have the same constrained domain $S(A,b)$,
\begin{eqnarray}
S(A,b)=\{( x^T, z^T)^T\in \mathbb R^{2n}| A x= b, \  \ - z\leq  x \leq  z, \  \
-r(A, b) e_n\leq x \leq r(A, b) e_n\}
\end{eqnarray}
Furthermore, we can conclude that the set $S(A,b)$ is a polygon because we can rewrite this set as the following
$S(A,b)=\{t | Q t \leq B\}$, where $t=( x^T, z^T)^T$,
\begin{equation}       
Q=\left(                 
  \begin{array}{cc}   
    A & \boldsymbol 0_{m\times n} \\  
    -A & \boldsymbol 0_{m\times n} \\
    I_n & -I_n \\
    -I_n & -I_n \\
    I_n & \boldsymbol 0_n  \\
    -I_n & \boldsymbol 0_n  \\
    \boldsymbol 0_n & -I_n \\
    \boldsymbol 0_n &  I_n
  \end{array}
\right)        
and \ B=\left(
\begin{array}{c}
\boldsymbol b \\
-\boldsymbol b \\
\boldsymbol 0_n \\
\boldsymbol 0_n \\
r(A,\boldsymbol b)\boldsymbol e_n \\
-r(A,\boldsymbol b)\boldsymbol e_n \\
\boldsymbol 0_n \\
r(A,\boldsymbol b)\boldsymbol e_n
\end{array}
\right)         
\end{equation}
It is obvious that 
\begin{eqnarray}
S(A,b)=S(A,b)\cap L(A,b)^{\perp}+L(A,b)
\end{eqnarray}
where $L(A,b)$ is the linear space of $S(A,b)$. On the other hand, we also can rewrite $S(A,b)$ as
\begin{eqnarray}
S(A,b)=\left\{\sum \limits_{i=1}^s\lambda_i\alpha_i+\sum \limits_{j=1}^wr_j\beta_j+\sum \limits_{k=1}^ul_k\gamma_j\bigg |\sum \limits_{i=1}^s\lambda_i =1,\ \lambda_i \geq 0, \ r_j\geq 0, \ l_k\in \mathbb{R}\right\}
\end{eqnarray}
Therefore, we have that 
\begin{eqnarray}
L(A,b)=span\{\gamma_1,...,\gamma_u\}
\end{eqnarray}
Let 
\begin{eqnarray}
E(A,b)=\{\alpha_1,...,\alpha_s\}
\end{eqnarray}

For any $t=(x^T,z^T)^T\in S(A,b)$, we define
\begin{eqnarray}
F_{p,q}(t)=\|z\|_{h_{p,q}}
\end{eqnarray}
It is obvious that $F_{p,q}(t)$ is a convex function and $F_{p,q}(t)\leq n\cdot h_{p,q}(r(A,b))$. By Lemma, we can get that 
\begin{eqnarray}
F_{p,q}(t)=F_{p,q}(t+\hat{t})
\end{eqnarray}
where $\hat{t} \in L(A,b)$.

By Lemma, 
\begin{eqnarray}
F_{p,q}(t)\leq F_{p,q}(t+\dot{t})
\end{eqnarray}
where $\dot{t}\in \left\{ \sum \limits_{j=1}^wr_j\beta_j \bigg | w_j\geq 0\right\}$
Therefore,
\begin{eqnarray}
\min \ F_{p,q}(t)\ \ s.t. \ t\in S(A,b)
\end{eqnarray}
is equal to
\begin{eqnarray}
\min \ F_{p,q}(t)\ \ s.t. \ t\in E(A,b)=\{\alpha_1,...,\alpha_s\}
\end{eqnarray}
For any $\alpha_i,\alpha_j\in E(A,b)$, if $\|\alpha_i\|_0\leq \|\alpha_j\|_0+1$ then there exists a constant $p_{\alpha_i,\alpha_j}$ such that $F_{p,q}(\alpha_i)<F_{p,q}(\alpha_j)$ whenever $0<p<p_{\alpha_i,\alpha_j}$.

By Lemma, if $\|\alpha_i\|_0= \|\alpha_j\|_0$ and $\mu(\alpha_i)\neq \mu(\alpha_j)$, then there exists a constant $p_{\alpha_i,\alpha_j}$ such that $F_{p,q}(\alpha_i)<F_{p,q}(\alpha_j)$ whenever $0<p<p_{\alpha_i,\alpha_j}$.

So for any we can define $p_{\alpha_i,\alpha_j}$
\begin{eqnarray}
p^*(A,b,q)=\min _{\alpha_i,\alpha_j\in E(A,b)} p_{\alpha_i,\alpha_j}
\end{eqnarray}
By the definition of $p^*(A,b,q)$ and $S(A,b)$, there exists a vector $\alpha_i\in E(A,b)$ such that 
\begin{eqnarray}
\|(\alpha_i)_{[n]}\|_{h_{p,q}}\leq \|x\|_{h_{p,q}}
\end{eqnarray}
for any $0<p<p^*_q(A,b)$ and any $x$ with $Ax=b$. Let $p$ tend to $0$, we can get that $\|(\alpha_i)_{[n]}\|_0\leq \|x\|_0$, i.e., $(\alpha_i)_{[n]}$ is the solution of both $l_{h_{p,q}}$-minimization and $l_0$-minimization.

The proof is complicated.
\end{proof}

As the first main theorem of this paper, we prove the equivalence under the linear equality constraint. In many applications, there always exists noises in the measurement vector $b$, so the responding model can be explained as the following,
\begin{eqnarray}\label{l0-e-model}
\min\limits_{x\in \mathbb{R}^n} \|x\|_{0} \ \ s.t. \|Ax-b\|_{f}\leq \varepsilon
\end{eqnarray}
where the function $\|\cdot\|$ is a certain norm function, and the following substitute model is
\begin{eqnarray}\label{lp-e-model}
\min\limits_{x\in \mathbb{R}^n} \|x\|_{h_{p,q}} \ \ s.t. \|Ax-b\|_{f}\leq \varepsilon
\end{eqnarray}
In the following corollaries, we show the equivalence relationship between when we consider $1$-norm and $\infty$-norm.
\begin{corollary}
For any given $A\in \mathbb{R}^{m\times n}$,$b\in \mathbb{R}^m$ and $0<q\leq 1$, there exists a constant $p^*_f(A,b,q)$ such that the solution of model (\ref{lp-e-model}) is the solution model (\ref{l0-e-model}) whenever $0<p<p^*(A,b,q)$.
\end{corollary}
\begin{proof}
Recall the proof of Theorem \ref{0123_0040}, the key component of proof is rewrite $Ax=b$ as a convex polytopes. Therefore it is enough to rewrite $\|Ax-b\|_1\leq \varepsilon $ and $\|Ax-b\|_{\infty}\leq \varepsilon$ as convex polygons.  

Consider the following set,
\begin{eqnarray}
\Omega=\{x\in \mathbb{R}^n|x_i\in \{0,1\},\ i\in[n]\}
\end{eqnarray}
It is obvious the set $\Omega$ is a finite set with $|\Omega|=2^n$, so let $\Omega=\{x^{(1)},x^{(2)},...,x^{(2^n)}\}$. Consider to construct a $\Phi=(x^{(1)},x^{(2)},...,x^{(2^n)})^T \in \mathbb{R}^{2^n\times n}$, then we can rewrite $\|Ax-b\|_1\leq \varepsilon$ as the following inequality,
\begin{eqnarray}
\Phi(Ax-b)\leq \varepsilon \boldsymbol 1
\end{eqnarray}
where $\boldsymbol 1 \in \mathbb{R}^n$ and $\Phi \in \mathbb{R}^{2^n\times n}$. 
It is easy to get that 
\begin{eqnarray}
\left\{x|\|Ax-b\|_{\infty}\leq \varepsilon\right\}
\end{eqnarray}
is equal to 
\begin{eqnarray}
-\varepsilon \boldsymbol 1+b\leq Ax \leq \varepsilon \boldsymbol 1+b
\end{eqnarray}
Similar to Lemma \ref{0731_1707}, it is obvious that the sub-matrix $A_{support(x^*)}$ is a full rank matrix, where $x^*$ is the solution of model (\ref{l0-e-model}) or model (\ref{lp-e-model}). Therefore, we can get that 
\begin{eqnarray}
\lambda_{min}\|x^*\|_2\leq \|Ax^*\|_2\leq C(f)\|Ax^*\|_f\leq C(f)\left(\|Ax^*\|_f+\|b\|_f\right)
\end{eqnarray} 
where $C(f)$ is a constant which based on the norm function $\|\cdot\|_f$ in model (\ref{lp-e-model}), so we can get that 
\begin{eqnarray}
0<\lambda_{min}=\min \limits_{T\in[n]} \lambda_{min}(A^T_TA_T) \ s.t. \ rank(A_T)=|T|
\end{eqnarray}
Therefore, it is obvious that the constraints in model (\ref{l0-e-model}) and model (\ref{lp-e-model}) are polytopes. Similar to the proof of Theorem \ref{0123_0040}, we can get the conclusion of this corollary.
\end{proof}
\subsection{The recovery condition of $l_{h_{p,q}}$-minimization}
There are a lot of paper focus on the recovery condition of alternative models. Except RIP, Null Space condition provides a sufficient and necessary condition for $l_1$-minimization. In this section, we present a sufficient and necessary condition for $l_{h_{p,q}}$-minimization. With presenting the matrix which satisfies such condition, we also give the stable result of $l_{h_{p,q}}$-minimization. Before the main theorem, the following lemmas are necessary. 
\begin{lemma}
For any fixed $0<q\leq 1$ and $0<p$, we have that
\begin{eqnarray}
h_{p,q}(x+y)\leq h_{p,q}(x)+h_{p,q}(y)
\end{eqnarray}
\end{lemma}
\begin{proof}
By the definition of $h_{p,q}()$, it is enough to consider the case when $x,y>0$. For any fixed $0<q\leq 1$, $0<p$ and $x>0$. let
\begin{eqnarray}
f(y)=h_{p,q}(x+y)-h_{p,q}(x)-h_{p,q}(y)
\end{eqnarray}
It is obvious that $f(0)=0$ and $f'(y)\leq 0$ since
\begin{eqnarray}
f'(y)=\frac{q(x+y)^{q-1}}{p+(x+y)^{q-1}}-\frac{qy^{q-1}}{p+y^q}
\end{eqnarray}
\end{proof}
Now, the following theorem presents a sufficient and necessary condition for $l_{h_{p,q}}$-minimization.
\begin{theorem}\label{0123_0205}
For any fixed $0<q\leq 1$ and $0<p$, if every $k$-sparse vector $x$ can be recovered by $l_{h_{p,q}}$-minimization then the following inequalities holds
\begin{eqnarray}\label{1218_1217}
\|x_{S}\|_{h_{p,q}}<\|x_{S^c}\|_{h_{p,q}}
\end{eqnarray}
for any non-zero vector $x\in N(A)$ and any index set $S\in [n]$ with $|S|\leq k$.
\end{theorem}
\begin{proof}
Assume $x^*$ is a $k$-sparse solution of $Ax=b$ and $\hat{x}$ is another solution of $Ax=b$, it is obvious that $x-\hat{x}\in N(A)$. If the inequality (\ref{1218_1217}) is satisfied, then we have that 
\begin{eqnarray}
\notag \|x^*+x\|_{h_{p,q}}&=&\|(x^*+x)_{S}\|_{h_{p,q}}+\|x_{S^c}\|_{h_{p,q}} \\
\notag &\geq &\|x^*\|_{h_{p,q}}-\|x_{S}\|_{h_{p,q}}+\|x_{S^c}\|_{h_{p,q}} \\
&\geq & \|x^*\|_{h_{p,q}}
\end{eqnarray}

On other hand, if there exists a vector $x\in N(A)$ and a index set $S$ such that
\begin{eqnarray}\label{1218_1710}
\|x_{S}\|_{h_{p,q}}\geq\|x_{S^c}\|_{h_{p,q}}
\end{eqnarray} 
It is obvious that $Ax_{S}=A(-x_{S^c})$. However the $k$-sparse vector $x_{S}$ can not be recovered by $l_{h_{p,q}}$-minimization.

\end{proof}
By Theorem \ref{0123_0205}, a sufficient and necessary condition is presented and it is also important to show that what matrices satisfies such condition. By a new concept named $h_{p,q}$-Null Space Constant (h-NSC), we will show the matrices which satisfy such condition.
\begin{definition}
For any fixed $0<q\leq 1$ and $0<p$, we define the $h_{p,q}$-Null Space Constant (h-NSC) $h_{p,q}(A,k)$ as the min number which satisfies the following inequalities
\begin{eqnarray}
\|x_{S}\|_{h_{p,q}}\leq h_{p,q}(A,k)\|x_{S^c}\|_{h_{p,q}}
\end{eqnarray}
for any non-zero vector $x\in N(A)$ and any index set $S\in [n]$ with $|S|\leq k$.
\end{definition}
By the definition of $h_{p,q}(A,k)$, it is easy to get the following corollary.
\begin{corollary}
For any fixed $0<q\leq 1$ and $0<p$, if $h_{p,q}(A,k)<1$ then every $k$-sparse vector $x$ can be recovered by $l_{h_{p,q}}$-minimization. 
\end{corollary}

Similar to the definition of $h_{p,q}(A,k)$, we can define $h_{\|\dot\|_1}(A,k)$ by changing $h_{p,q}(\cdot)$ into $\|\cdot\|_1$. Now, we will show some proposition of $h_{p,q}(A,k)$.
\begin{proposition}
For any fixed $0<q\leq 1$ and $0<p$, we have that 
\begin{eqnarray}
h_{p,q}(A,k)\leq h_{\|\dot\|_1}(A,k)
\end{eqnarray}
\end{proposition}
\begin{proof}
For a given vector $\beta\in N(A)$, we denote $S_{\beta,k}$ as the index set of the $k$ largest elements,
\begin{eqnarray}\label{1219_1615}
\theta (\beta,p,q,k)=\frac{\|\beta_{S_{\beta,k}}\|_{h_{p,q}}}{\|\beta_{S_{\beta,k}^C}\|_{h_{p,q}}}
\end{eqnarray}
It is easy to get that
\begin{eqnarray}
h_{p,q}(A,k)=\max \limits_{\beta\in N(A)} \theta (\beta,p,q,k)
\end{eqnarray}

We notice that the function $\frac{log(1+x^q/p)}{x}$ is an decreasing function when $x>0$, so we can get that 
\begin{eqnarray}
\frac{h_{p,q}(\beta_j)}{|\beta_j|}\geq \frac{h_{p,q}(\beta_i)}{|\beta_i|}
\end{eqnarray}
where $i\in S_{\beta,k}$ and $j\in S_{\beta,k}^C$.

Rewrite the above inequality, we can get that 
\begin{eqnarray}
\frac{h_{p,q}(\beta_j)}{h_{p,q}(\beta_i)}\geq \frac{|\beta_j|}{|\beta_i|}
\end{eqnarray}
\begin{eqnarray}
\frac{\|\beta_{S_{\beta,k}}\|_{h_{p,q}}}{h_{p,q}(\beta_i)}\geq \frac{\|\beta_{S_{\beta,k}}\|_1}{|\beta_i|}
\end{eqnarray}
\begin{eqnarray}
\frac{\|\beta_{S_{\beta,k}}\|_{h_{p,q}}}{\|\beta_{S_{\beta,k}^C}\|_{h_{p,q}}}\leq \frac{\|\beta_{S_{\beta,k}}\|_1}{\|\beta_{S_{\beta,k}^C}\|_1}
\end{eqnarray}
Therefore, we have that 
\begin{eqnarray}
h_{p,q}(A,k)\leq h_{\|\dot\|_1}(A,k)
\end{eqnarray}
\end{proof}
\begin{corollary}
If every $k$-sparse vector can be recovered by $l_1$-minimization, then they also can be recovered by $l_{h_{p,q}}$-minimization for any $p>0$ and $0<q\leq 1$.
\end{corollary}
\begin{proof}
In order to prove this corollary, we just need to prove that if every $k$-sparse vector can be recovered by $l_1$-minimization then $h_{\|\dot\|_1}(A,k)<1$.

Similar to (\ref{1219_1615}), we define the following function
\begin{eqnarray}
\theta (\beta,\|\dot\|_1,k)=\frac{\|\beta_{S_{\beta,k}}\|_{1}}{\|\beta_{S_{\beta,k}^C}\|_{1}}
\end{eqnarray}
It is obvious that $\theta (\beta,\|\dot\|_1,k)$ is a linear function, so we can get that
\begin{eqnarray}
h_{\|\dot\|_1}(A,k)=\max \limits_{\beta\in N(A)\cap B_{\|\dot\|_1}} \theta (\beta,p,q,k)
\end{eqnarray}
where $ B_{\|\dot\|_1}$ is the 1-norm unit ball. Because the set $ N(A)\cap B_{\|\dot\|_1}$ is a complex set, there exist $\beta ^* \in N(A)\cap B_{\|\dot\|_1}$ such that $h_{\|\dot\|_1}(A,k)=\theta (\beta^*,p,q,k)$.  
\end{proof}
\begin{theorem}
For any matrix $A\in \mathbb{R}^{m\times n}$ which satisfies $2k$-order RIP and $\delta_{2k}\leq \frac{\sqrt{2}-1}{2}$, then any $k$-sparse vector can be recovered by $l_1$-minimization
\end{theorem}
\begin{theorem}
Let $A\in \mathbb{R}^{m\times n}$ be a random draw of a Gaussian matrix. Let $k\leq n$, $\rho \in (0,1]$, $\epsilon in (0,1)$ such that
\begin{eqnarray}
D(x)=\frac{1}{\sqrt{2\log(ex^{-1})}}+\frac{1}{8\pi^3\log (ex^{-1})}
\end{eqnarray}
\begin{eqnarray}
\frac{m^2}{m+1}\geq 2k\log(ex^{-1})\left(1+\rho^{-1}+D(k/n)+\sqrt{\frac{\log(\epsilon ^{-1})}{k\log(en/k)}}\right)
\end{eqnarray}
Then with probability at least $1-\epsilon$ the matrix A satisfies the stable null space
property of order s with $h_{\|\cdots\|_1}=\rho$ 
\end{theorem}
\begin{theorem}
If $h_{p,q}(A,k)\leq 1$, for any $x^*\in \mathbb{R}^n$ we have that 
\begin{eqnarray}
\|x^*-\vartriangle _{h_{p,q}}(Ax^*)\|_{h_{p,q}}\leq \frac{2(1+h_{p,q}(A,k))}{1-h_{p,q}(A,k)}\|\sigma(x)\|_{h_{p,q}}
\end{eqnarray}
where $\vartriangle _{h_{p,q}}(Ax^*)$ is the solution of $l_{h_{p,q}}$-minimization with $b=Ax^*$.
\end{theorem}
\begin{proof}
Let $y=\vartriangle _{h_{p,q}}(Ax^*)$, it is obvious that $v=x-y\in Ker(A)$ and
\begin{eqnarray}
\|v_S\|_{h_{p,q}}\leq h_{p,q}(A,k)\|v_{S^C}\|_{h_{p,q}}
\end{eqnarray}
for any $S\in [n]$ with $|S|\leq k$.
\begin{eqnarray}
\notag \|x\|_{h_{p,q}}&=&\|x^S\|_{h_{p,q}}+\|x_{S^C}\|_{h_{p,q}} \\
&\leq & \|x_{S^C}\|_{h_{p,q}}+\|(x-y)_{S}\|_{h_{p,q}}+\|y_{S}\|_{h_{p,q}}
\end{eqnarray}
\begin{eqnarray}
\notag \|v_{S^C}\|_{h_{p,q}} & \leq & \|x_{S^C}\|_{h_{p,q}}+\|y_{S^C}\|_{h_{p,q}} \\
\notag &=& \|y\|_{h_{p,q}}-\|y_{S}\|_{h_{p,q}}+\|x_{S^C}\|_{h_{p,q}}  \\
\notag &\leq & \|y\|_{h_{p,q}}+2\|x_{S^C}\|_{h_{p,q}}-\|x\|_{h_{p,q}}+\|v_{S}\|_{h_{p,q}} \\
&\leq & \|y\|_{h_{p,q}}-\|x\|_{h_{p,q}}+\|x_{S^C}\|_{h_{p,q}}+h_{p,q}(A,k)\|v_{S^C}\|_{h_{p,q}}
\end{eqnarray}
Therefore,
\begin{eqnarray}
\|v_{S^C}\|_{h_{p,q}} \leq \frac{1}{1-h_{p,q}(A,k)}\left(\|y\|_{h_{p,q}}-\|x\|_{h_{p,q}}+2\|x_{S^C}\|_{h_{p,q}}\right)
\end{eqnarray}
Therefore,
\begin{eqnarray}
\|v_{S^C}\|_{h_{p,q}}\leq \frac{2}{1-h_{p,q}(A,k)}\|x_{S^C}\|_{h_{p,q}}
\end{eqnarray}
\begin{eqnarray}
\notag \|v\|_{h_{p,q}}&=&\|v_S\|_{h_{p,q}}+\|v_{S^C}\|_{h_{p,q}} \\
\notag &\leq & (1+h_{p,q}(A,k))\|v_{S^C}\|_{h_{p,q}} \\
&\leq & \frac{2(1+h_{p,q}(A,k))}{1-h_{p,q}(A,k)}\|x_{S^C}\|_{h_{p,q}}
\end{eqnarray}
Take $S$ as the index of the largest $k$ elements of $x$, so we can get that 
\begin{eqnarray}
\|x^*-\vartriangle _{h_{p,q}}(Ax^*)\|_{h_{p,q}}\leq \frac{2(1+h_{p,q}(A,k))}{1-h_{p,q}(A,k)}\|\sigma(x)\|_{h_{p,q}}
\end{eqnarray}
\end{proof}
\section{The local optimal property and an unify algorithm for $l_{h_{p,q}}$-minimization}
In this section, we will focus on the application of $l_{h_{p,q}}$-minimization. Although $h_{p,q}(\cdot)$ is not a smooth function, we will show an analysis expression of its local optimal solution. By such analysis expression, a fixed point iterative algorithm is presented and we also prove the convergency of this algorithm.
\subsection{The local optimal property of $l_{h_{p,q}}$-minimization}
For a given $x\in \mathbb{R}^n$, we define two diagonal matrices $H(x),F(x)\in \mathbb{R}^{n\times n}$,
\begin{eqnarray}
H(x)_{i,i}=
\begin{cases}
\frac{q}{|x_i|^{2-q}(p+|x_i|^q} &  |x_i|\neq 0 \\
0 & |x_i|= 0
\end{cases}
\end{eqnarray}
and 
\begin{eqnarray}
F(x)_{i,i}=\frac{|x_i|^{2-q}(p+|x_i|^q}{q}
\end{eqnarray}
With $H(x)$ and $F(x)$, the following theorem will show the local property of $l_{h_{p,q}}$-minimization.
\begin{theorem}
If the underdetermined matrix $A\in \mathbb{R}^{m\times n}$ satisfies that $rank(A)=m$, $x^*$ is the solution of $l_{h_p}$-minimization (\ref{hp-model}), then $x^*$ satisfies the following equation,
\begin{eqnarray}\label{0131_1731}
x^*=F(x^*)A^T(AF(x^*)A^T)^{-1} b
\end{eqnarray}
\end{theorem}
\begin{proof}
For the solution $x^*\in \mathbb{R}^n$, we assume that $\|x^*\|_0=s$, then it is obvious that there exists an elementary orthogonal matrix $E(x^*)\in \mathbb{R}^{n\times n}$ such that 
\begin{eqnarray}
S=support(\widetilde{x})=[s]
\end{eqnarray}
where $\widetilde{x}=E(x^*)x^*$ and $E(x^*)=E(x^*)^{-1}$.

Let $z=\widetilde{x}_{[s]}$ and consider the following problem,
\begin{eqnarray}\label{0620_2107}
\notag & \min\limits_{t\in \mathbb{R}^s} \|t\|_{h_{p,q}}\\
&s.t. Bt=b
\end{eqnarray}
where $B=(AE(x^*))_{[s]}\in \mathbb{R}^{m\times s}$. It is obvious that $z$ is the solution of model (\ref{0620_2107}). We notice that $support(z)=[s]$ so there exists a constant $\eta$ small enough such that the function $\|\dot\|_{h_{p,q}}$ is differentiable at the point $x=z$ when $\|t-z\|_2\leq \eta $. Therefore, KKT condition can be applied in such area. Define the lagrange function $L(t,\lambda)$ as 
\begin{eqnarray}
L(t,\lambda)=\|t\|_{h_{p,q}}-\lambda ^T(Bt-b)
\end{eqnarray}
Therefore, $z$ must be the solution of the following equations,
\begin{eqnarray}\label{0620_2306}
\begin{cases}
\frac{\partial L}{\partial x}|_{x=z}=0 \\
Bz=b
\end{cases}
\end{eqnarray}
It is obvious that 
\begin{eqnarray}
h_p'(x_i)=\frac{sign(x_i) q \|x_i\|^{q-1}}{(p+|x_i^q|)}=\frac{q x_i}{|x_i|^{2-q}(p+|x_i|^q)}
\end{eqnarray}
Therefore, we can rewrite the equations (\ref{0620_2306}) as
\begin{eqnarray}\label{0620_2307}
\begin{cases}
H(z)z-B^T\lambda ^*=0 \\
Bz=b
\end{cases}
\end{eqnarray}
Since $support(z)=[s]$, $H(z)\in \mathbb{R}^{s\times s}$ is an invertible matrix and $H(z)^{-1}=F(z)$. By the equations (\ref{0620_2307}), we can get that 
\begin{eqnarray}
BH(z)^{-1}H(z)z=AH(z)^{-1}A^T\lambda ^*
\end{eqnarray}
therefore, 
\begin{eqnarray}
\lambda ^*= \left( BF(z)B^T \right)^{-1}b 
\end{eqnarray}
So, we can conclude that 
\begin{eqnarray}
z=F(z)B^T(BF(z)B^T)^{-1}b.
\end{eqnarray}
By the definition, it is obvious that 
\begin{equation}       
F(\widetilde{x})=\left(                 
  \begin{array}{cc}   
    F(z) & \boldsymbol 0_{s\times (n-s)} \\  
    \boldsymbol 0_{(m-s)\times (n-s)} & \boldsymbol 0_{s\times s}
  \end{array}
\right)           
\end{equation}
 and it is easy to get that
\begin{eqnarray}
&&AE(x^*)F(\widetilde{x})(AE(x^*))^T \\
&=&\left( B \ \ (AE(x^*))_{[n]\setminus [s]}\right) 
\left(                 
  \begin{array}{cc}   
    F(z) & \boldsymbol 0_{s\times (n-s)} \\  
    \boldsymbol 0_{(m-s)\times (n-s)} & \boldsymbol 0_{s\times s}
  \end{array}
\right)           
\left(
\begin{array}{c}
B^T \\
(AE(x^*))_{[n]\setminus [s]}^T
\end{array}
\right) \\
&=& BF(z)B^T 
\end{eqnarray}
And
\begin{eqnarray}
\notag F(\widetilde{x})(AE(x^*))^T&=&\left(                 
  \begin{array}{cc}   
    F(z) & \boldsymbol 0_{s\times (n-s)} \\  
    \boldsymbol 0_{(m-s)\times (n-s)} & \boldsymbol 0_{s\times s}
  \end{array}
\right)           
\left(
\begin{array}{c}
B^T \\
(AE(x^*))^T
\end{array}
\right) \\
&=& \left(
\begin{array}{c}
F(z)B^T \\
\boldsymbol 0
\end{array}
\right)
\end{eqnarray}
Therefore, we can conclude that
\begin{eqnarray}
\widetilde{x}=F(\widetilde{x})(AE(x^*))^T(AE(x^*)F(\widetilde{x})(AE(x^*))^T)^{-1}b
\end{eqnarray}
Since $\widetilde{x}=E(x^*)x^*$, it is easy to get that 
\begin{eqnarray}
F(\widetilde{x})=E(x^*)F(x^*)E(x^*)
\end{eqnarray}
It is easy to get that 
\begin{eqnarray}
x^*=F(x^*)A^T(AF(x^*)A^T)^{-1}b
\end{eqnarray}
\end{proof}
\begin{algorithm}[H]
\caption{An unify algorithm for $l_{h_p}$-minimization}
\label{alg1}
\begin{algorithmic}
\REQUIRE $A\in \mathbb{R}^{m \times n}, \quad  b \in \mathbb{R}^n, \quad f_p(x), \quad p .$
\ENSURE $ x^*.$
\STATE $ x^1 = arg \min \limits_{Ax=b} \| x\|_1$
\STATE $D=D(\boldsymbol x^1)$
\FOR{$k=1,2,\cdots$ until convergence}
\STATE $ x^{k+1}=DA^T(ADA^T)^{\dagger} b$
\STATE $D=D( x^{k+1})$
\STATE $k=k+1$
\ENDFOR
\STATE $ x^*= x^{k+1}$
\end{algorithmic}
\end{algorithm}
By the analysis expression \ref{0131_1731}, a fixed point iterative algorithm is presented in Algorithm \ref{alg1}. Before we give the convergency result of the algorithm \ref{alg1}, some lemmas are necessary.
\begin{lemma}\label{1101_2336}
For a fixed $\alpha>0$, we have that
\begin{eqnarray}
h_{p,q}(x)-\frac{h_{p,q}'(\alpha)}{2\alpha}x^2\leq h_{p,q}(\alpha)-\frac{\alpha h_{p,q}'(\alpha)}{2}
\end{eqnarray} 
for any $x\geq 0$
\end{lemma}
\begin{proof}
Define a function $g(x)$ as followings
\begin{eqnarray}
g(x)=h_{p,q}(x)-\frac{h_{p,q}'(\alpha)}{2\alpha}x^2
\end{eqnarray}
It is obvious that $g(0)=0$ and
\begin{eqnarray}
g'(x)=h_{p,q}'(x)-\frac{xh_{p,q}'(\alpha)}{\alpha}
\end{eqnarray}
\begin{eqnarray}
g''(x)=h_{p,q}''(x)-\frac{h_{p,q}'(\alpha)}{\alpha}\leq 0
\end{eqnarray}
Therefore $g'(x)$ is a non increasing function. Since $g'(\alpha)=0$, we have that
\begin{eqnarray}
g(x)\leq g(\alpha)= h_{p,q}(\alpha)-\frac{\alpha h_{p,q}'(\alpha)}{2}
\end{eqnarray}
\end{proof}
Now, we give the convergence conclusion of Algorithm \ref{alg1}.
\begin{theorem}
The sequence $\{x^{n}\}$ get by Algorithm \ref{alg1} satisfies the following inequality,
\begin{eqnarray}
\|x^{k+1}\|_{h_{p,q}}\leq \|x^k\|_{h_{p,q}}
\end{eqnarray}
and the limit point $x^*$ of $\{x^{n}\}$ satisfies the following the equality,
\begin{eqnarray}
x^*=F(x^*)A^T(AF(x^*)A^T)^\dagger b
\end{eqnarray}
\end{theorem}
\begin{proof}
At ($k+1$)th iteration, we need to point out that $ x^{k+1}$ is the solution of the following model
\begin{eqnarray}
\min \limits_{Ax=b} \frac{1}{2} x^TH(x^k) x^T.
\end{eqnarray}
With the convex quadratic programming above, we consider its Lagrange dual function,
\begin{equation*}
L( x,  \lambda)=\frac{1}{2} x^TH( x^k) x^T- \lambda^T(A x- b).
\end{equation*}
Therefore, the optimal solution $( x^*, \lambda ^*)$ should satisfy the following equations
\begin{equation}
\left\{
\begin{array}{l}
\frac{\partial L}{\partial  x}=H(x^k) x-A \lambda ^T =\boldsymbol 0,\\
A x= b.
\end{array}
\right.
\end{equation}
Hence, we find that
\begin{eqnarray}\label{1128_0017}
( x^{k+1})^TH( x^k) x^{k+1} \leq ( x^{k})^TH( x^k) x^k.
\end{eqnarray}
i.e.,
\begin{eqnarray}\label{1101_2242}
\sum \limits_{x_i^k\neq 0}\frac{h_{p,q}'(x_i^k)}{|x_i^k|}|x_i^{k+1}|^2\leq \sum \limits_{x_i^k\neq 0}\frac{h_{p,q}'(x_i^k)}{|x_i^k|}|x_i^{k}|^2
\end{eqnarray}
By Lemma \ref{1101_2336}, if $x_i^k\neq 0$, then we have that
\begin{eqnarray}\label{1101_2245}
h_p(x_i^{k+1})-\frac{h_p'(x_i^k)}{2x_i^k}\leq h_p(|x_i^k|)-\frac{|x_i^k|h_p'(|x_i^k|)}{2}
\end{eqnarray}
we deduce that
\begin{equation}
\|x^{k+1}\|_{h_p}\leq \|x^{k}\|_{h_p}
\end{equation}
by employing (\ref{1101_2242}) and (\ref{1101_2245}).
So far, we have proved that the objective function decrease at the series $\{x^{k}\}$ points. 
Therefore the convergent point $ x^*$ is the solution of the following problem
\begin{eqnarray}
\min \limits_{A  x= b} \frac{1}{2} x^TG( x^*) x^T.
\end{eqnarray}
Repeat the discussion above, the proof is completed.
\end{proof}

\section{Multiple source location problem by TDOA}
As we have introduced in Section 1, multiple source location problem can be modelled by sparse point represent method. Suppose there exist K targets and m receivers, and $k$-th ($k=1,2,...,K$) target broadcasts a time domain signal $s_k(t)$. Then the signal received by $j$-th receiver can be expressed as 
\begin{eqnarray}
h_j(t)=\sum\limits_{i=1}^Kp_{j,i}s_i(t-t_{j,i})+n_j(t)
\end{eqnarray}
where $n_j(t)$ is the noise, $p_{j,i}$ is the channel coefficient and $t_{j,i}$ stands the time delay from $i$-th target to $j$-th source. Without of generality, the signals $s_k(t)$ and $h_j(t)$ are assumed to be ergodic, mutually uncorrelated white sequences. Therefore, we can get that
\begin{eqnarray}
\int _{t} s_i(t)s_{i'}(t-a)dt=
\begin{cases}
0 & i\neq i' \\
\delta (a) & i=i'
\end{cases}
\end{eqnarray}
\begin{eqnarray}
\int _{t} n_i(t)n_{i'}(t-a)dt=
\begin{cases}
0 & i\neq i' \\
\delta (a) & i=i'
\end{cases}
\end{eqnarray}
\begin{eqnarray}
\int_{t} s_i(t)n_{i'}(t)dt=0
\end{eqnarray}
If we take $1$-st receiver as the reference, then it is easy to get that
\begin{eqnarray}
\notag r(x)&=&\int_t h_j(t)\times h_1(t) \\
\notag &=& \int_t \left( \left(\sum\limits_{i=1}^Kp_{j,i}s_i(t-t_{j,i})+n_j(t)\right)\times \left(\sum\limits_{i=1}^Kp_{1,i}s_i(t-t_{1,i})+n_1(t)\right)\right) \\
&=& \sum\limits_{i=1}^Kp_{j,i}p_{1,i}\delta \left(x-(t_{j,i}-t_{1,i})\right)
\end{eqnarray}
If there are $m+1$ receivers, then we can get a data matrix $L$,
\begin{eqnarray}
L=\left(
\begin{array}{cccc}
l_1^1 & l_1^2 & \cdots & l_1^K \\
l_2^1 & l_2^2 & \cdots & l_2^K \\
\cdots & \cdots & \cdots &\cdots \\
l_m^1 & l_m^2 & \cdots & l_m^K \\
\end{array}
\right)
\end{eqnarray}
Define $\varphi _u(x)=x^u$, where $u=1,2,...,U\leq K$. The multiple source location problem can be modelled as the following  $l_{0}^\eta (\varepsilon)$-minimization
\begin{eqnarray}\label{0127_0248}
\notag & \min\limits_{x\in \mathbb{R}^n} \|x\|_0\\
&s.t.\quad
\begin{cases}
\|Ax-b\|_{\infty}\leq \varepsilon \\
\|x\|_{\infty}\leq \eta.
\end{cases}
\end{eqnarray}
where $A=\left[\left(A^1\right)^T \ \left(A^2\right)^T \ \left(A^3\right)^T\cdots \left(A^U\right)^T\right]^T\in \mathbb{R}^{Um\times n}$, $A^u\in \mathbb{R}^{m\times n}$, 
\begin{eqnarray}
(A^u)_{i,j}=\varphi (\|w_j-T_i\|)
\end{eqnarray}
and $b=\left[\left(b^1\right)^T \ \left(b^2\right)^T \ \left(b^3\right)^T\cdots \left(b^U\right)^T\right]^T\in \mathbb{R}^{Um}$, $b^u\in \mathbb{R}^m$, 
\begin{eqnarray}
(b^u)_{i}=\sum\limits_{k=1}^K\varphi _u(l_i^k)
\end{eqnarray}
the constant $\varepsilon$ stands the error caused by noise and off-grid cases while the constant $\eta$ is larger than $1$.
\begin{remark}
In the model (\ref{0127_0248}), the elements of solution $x^*$ should be one or zeros in ideal situation and we adopt $\|x\|_{\infty}\leq \eta$ to keep the elements of solution in a reasonable range.
\end{remark}
\begin{remark}
The reason why we adopt the function $\varphi(\cdot)$ comes from the Theorem. It is obvious that these functions $\varphi_u(\cdot)$ are linearly independent by Theorem, i.e., there only exist an unique data matrix can provide such $b$ in model (\ref{0127_0248}), so we increase the rank of measurement matrix with the increasing of row vectors. 
\end{remark}
 
\subsection{Sparse model for multiple source location problem}
As we have discussed above, the sparse model \ref{0127_0248} is used for solving this problem, so we consider to use $h_{p,q}(\cdot)$ to instead of $0$-norm, so the alternative model is the following optimization problem and some simple analysis about this model is necessary to be presented, 
\begin{eqnarray}\label{0127_0320}
\notag & \min\limits_{x\in \mathbb{R}^n} \|x\|_{h_{p,q}}\\
&s.t.\quad
\begin{cases}
\|Ax-b\|_{\infty}\leq \varepsilon \\
\|x\|_{\infty}\leq \eta.
\end{cases}
\end{eqnarray}
\begin{theorem}
For given $\eta$, $\varepsilon$, $A$ and $b$, there exist a constant $p^*(A,b,\varepsilon,\eta)$ such that the solution of model (\ref{0127_0320})is the solution of model (\ref{0127_0248}) whenever $0<p<p^*(A,b,\varepsilon,\eta)$.
\end{theorem}
\begin{proof}
Both of these two optimal problems' constraint region can be rewritten as the following
\begin{eqnarray}
\left(
\begin{array}{c}
A \\
-A \\
I \\
-I
\end{array}
\right)
x\leq
\left(
\begin{array}{c}
\varepsilon \boldsymbol 1+b \\
\varepsilon \boldsymbol 1-b \\
\eta \boldsymbol 1 \\
-\eta \boldsymbol 1
\end{array}
\right)
\end{eqnarray}
Similar to the proof of Theorem and Corollary, we can get the conclusion of this theorem.
\end{proof}
Different to Algorithm 1, the model (\ref{0127_0320}) has inequality constraints. Inspired by OMP, we present an improved algorithm based Algorithm 1. 
\begin{algorithm}[H]                           
\caption{Improved algorithm for $l_{h_p}$-minimization}
\label{alg2}
\begin{algorithmic}
\REQUIRE $A\in \mathbb{R}^{m \times n}, \quad  b \in \mathbb{R}^n, \quad p, \quad p$ and set $V_a$
\ENSURE $ x^*.$
\STATE $ x^1 = arg \min \limits_{A x= b} \| x\|_1$
\STATE $D=D( x^{1})$
\FOR{$k=1,2,\cdots$ until the stopping criterion is satisfied}
\STATE $ x^{k+1}_1=DA^T(ADA^T)^{\dagger}b$
\STATE $S=\{i||x^{k+1}_i|\in Va\}$
\STATE $ x^{k+1}=arg\min \limits_{supp( x)\subseteq  S} \|A x- b\|_2^2$
\STATE $D=D( x^{k+1})$
\STATE $k=k+1$
\ENDFOR
\STATE $ x^*= x^{k+1}$
\end{algorithmic}
\end{algorithm}
For the set $Va$ which is used to control the value of solution since the element of the solution is one either zero. Usually, we can let $Va=[0.2,1.2]$. Next, we will introduce an method to update the grid points to ensure the location of targets belongs to the grid points. In order to get the ideal solution, we need to introduce the following function $f_a(\cdot)$,
\begin{equation}
f_a(x)=
\begin{cases}
|x|/a & |x|\leq a \\
|1-|x|/(1-a)| & else
\end{cases}
\end{equation}
It is obvious that the function $f_a(\cdot)$ gets its minimize value at points $x=1$ or $x=0$ and its max point at $x=a$, so the value of ideal solution of model (\ref{0127_0320}) should be zero. Since we do not know the location of targets and the targets may be not locate at grid points, so we get calculate the value of the solution to judge whether the grid points contain the locations. Once the value of a certain element is too large, it means the corresponding grid points should be corrected. Once we have identified the grid point which needs to be modified, With the update method $\varphi _q$, the modified point should be closer to the real location than the original point. Generally, the update method $\varphi _q$ can be various forms, in this paper, we present an update method $\varphi _q$ as following. For a fixed point $w_i$, we consider eight alternative vectors $z_j$, 
\begin{eqnarray}
z_j=w_i+\frac{s}{2^q}\left(cos\left(\frac{\pi(j-1)}{4}\right),sin\left(\frac{\pi(j-1)}{4}\right)\right)
\end{eqnarray}
where $j=1,2,...,8$. Let $\Xi (w_i)=\{w_i,z_1,z_2,...z_8\}$. Now we can give the definition of the update method $\varphi _q$,
\begin{eqnarray}
\varphi _q(w_i)=\arg \min \limits_{z\in \Xi (w_i) } \|x^*\|_{f_a}\ x^* \ is \ the \ solution \ of \ model\ (\ref{0127_0320}), \ where \ w_i=z 
\end{eqnarray} 
To summarize, we present the whole process of solving multiple source location problem. 
\begin{description}
	\item[Step 1] Input the matrix $T=(T_1^T,T_2^T,...,T_m^T)^T\in \mathbb{R}^{m+1\times 2}$, where $T_i$ is the location of i-th sensor. The grid point matrix $W=(w_1^T,w_2^T,...,w_n^T)^T\in \mathbb{R}^{n\times 2}$, where $w_i$ is the location of i-th grid point. The measurement matrix $L\in \mathbb{R}^{(m-1)\times k}$, the value of $p,a$, the number of iterative $G$, and the thresholding value $\delta,\varepsilon >0$.
	\item[Step 2] According to model (\ref{0127_0320}), we can obtain the corresponding $l_{h_p}$-minimization.
	\item[Step 3] By Algorithm \ref{alg2}, we can obtain the solution $x^*$.
	\item[Step 4] If $\|x^*\|_{f_a} \geq \delta $,	Denote the index set 
\begin{eqnarray}
V=\{i| f(|x^*_i|)\geq \varepsilon\}
\end{eqnarray}	
and 
\begin{eqnarray}
W=\{w_i|i\in V\}=\{w_{j_1},w_{j_2},...,w_{j_{|V|}}\}
\end{eqnarray}
    \item[Step 5] For $i=1:1:|V|$ 
                   
                   \ \ \ \ For  $q=1:1:G$
                  
                  Define the update method $\varphi _q$, and $w_{j_i}=\varphi _q(w_{j_i})$. 
                  
                  According to model (\ref{0127_0320}), we can obtain the  corresponding new $l_{h_{p,q}}$-minimization and its solution $x^*$. If $\|x\|_{f_a}\leq \delta$, then break and turn to Step 6
                  
                  \ \ \ \ end
                  
                  \ \ \ \ \ \ end
	
	\item[Step 6] By the solution $x^*$, we can get the result of location, i.e., $w_j$ with $j\in V^*$, where $V^*$ is the $k$-largest absolute value of $x^*$.
\end{description}
\subsection{Simulation Experiment}
In this subsection, we present some simulation experiment to show the validity of our method. In all experiments, we consider a $10km\times 10km$ square zone which contains all the targets. In Algorithm 2, we set $p=0.1$, $q=1$, $Va=[0.2,1.2]$ and the maximum number of iteration is $30$. In the update method, we set $a=0.5$, $G=2$, $\varepsilon \geq 0.3$ and $\delta=0.3 *K$, where $K$ is the number of targets. The location of receivers are randomly picked in the square zone.

In Figure \ref{fig5-5}, we show the detail processing of the update method $\varphi _q$. There exist three targets and none of them is located on the gird points. With the update method, we find that the modified grid point is close to the real location.
\begin{figure}[!ht]
	\centering
	\subfigure[]{\includegraphics[width=0.325\textwidth]{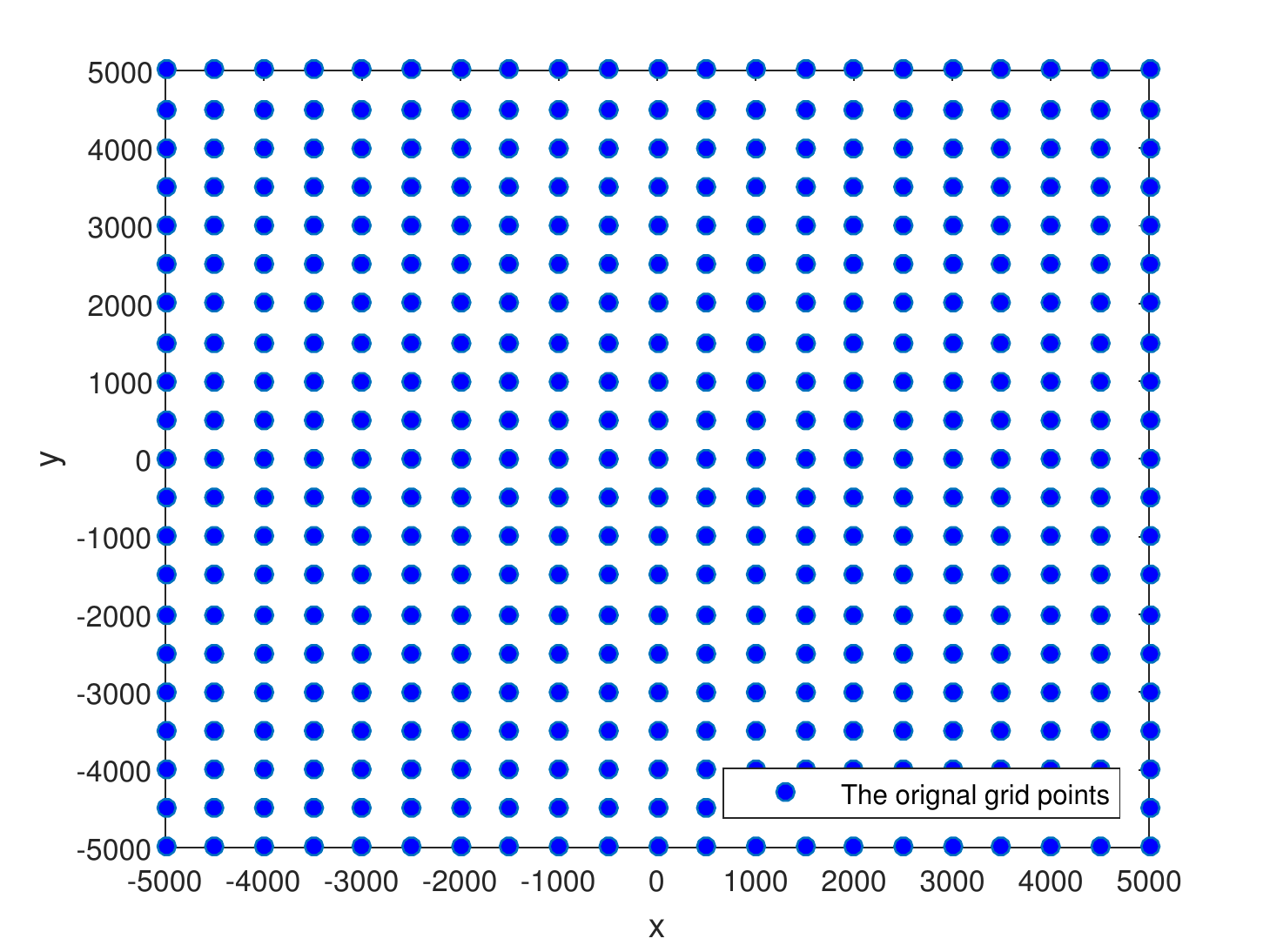}} \hfill
	\subfigure[]{\includegraphics[width=0.325\textwidth]{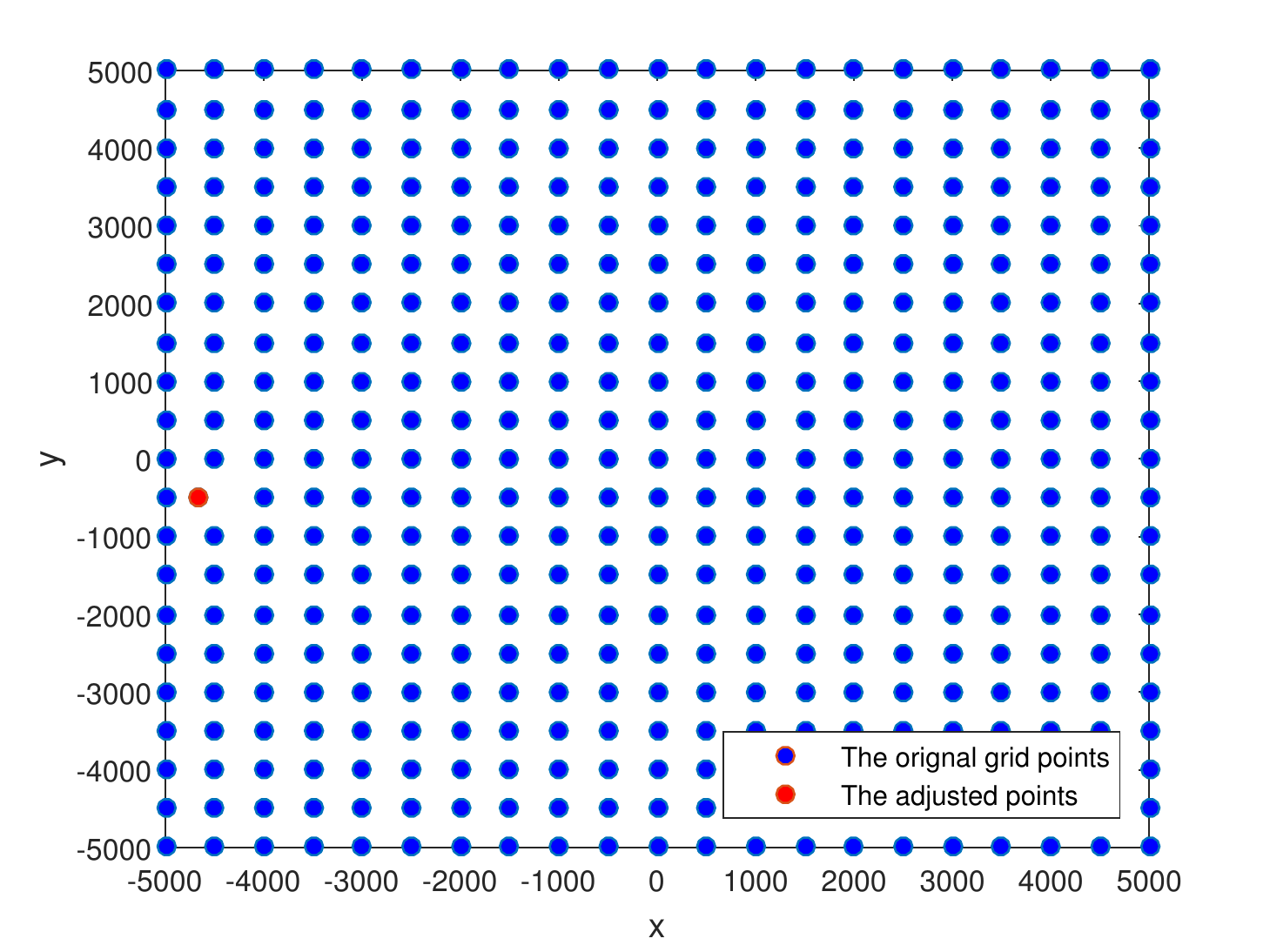}}
	\subfigure[]{\includegraphics[width=0.325\textwidth]{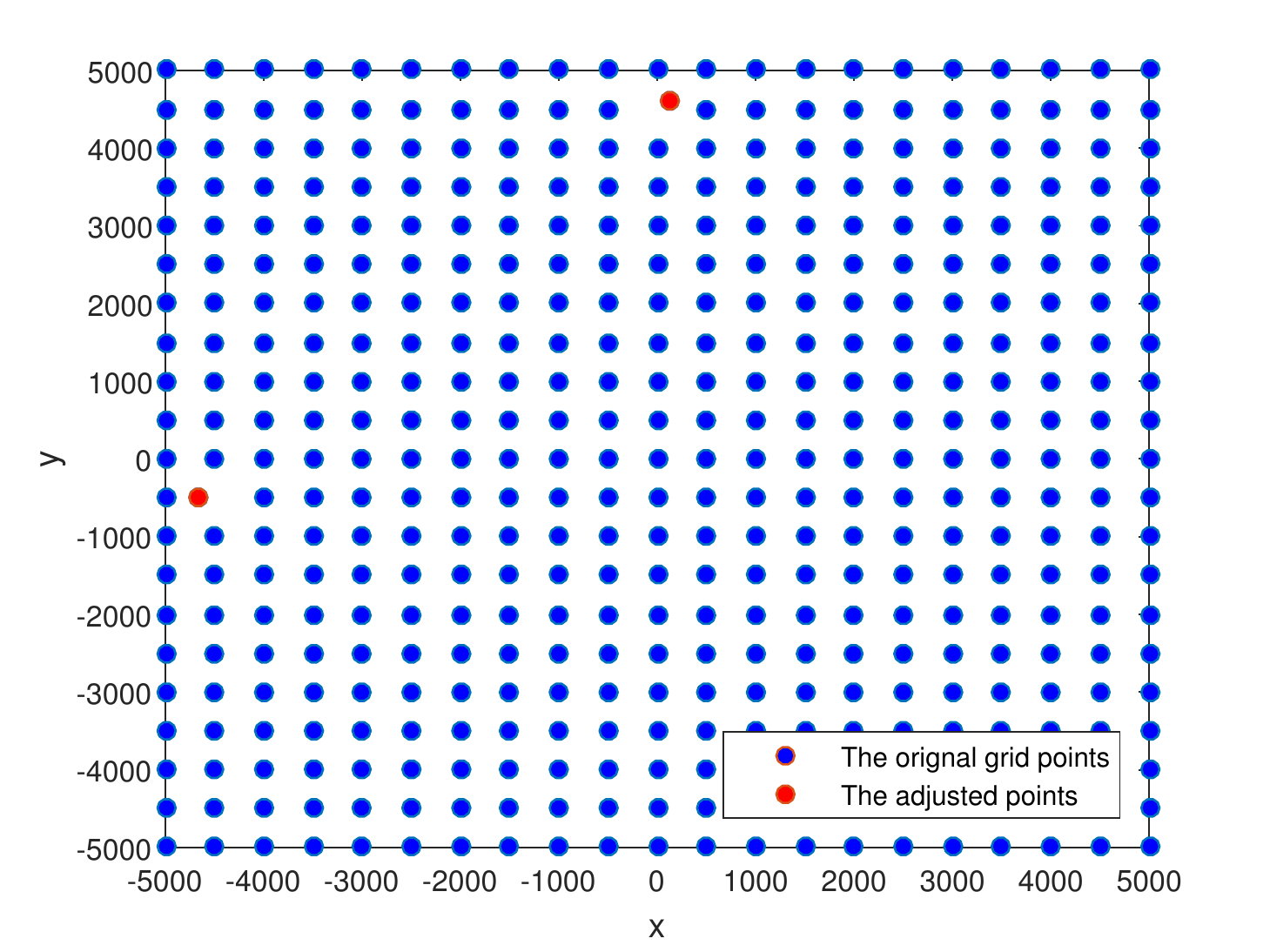}}
	\subfigure[]{\includegraphics[width=0.325\textwidth]{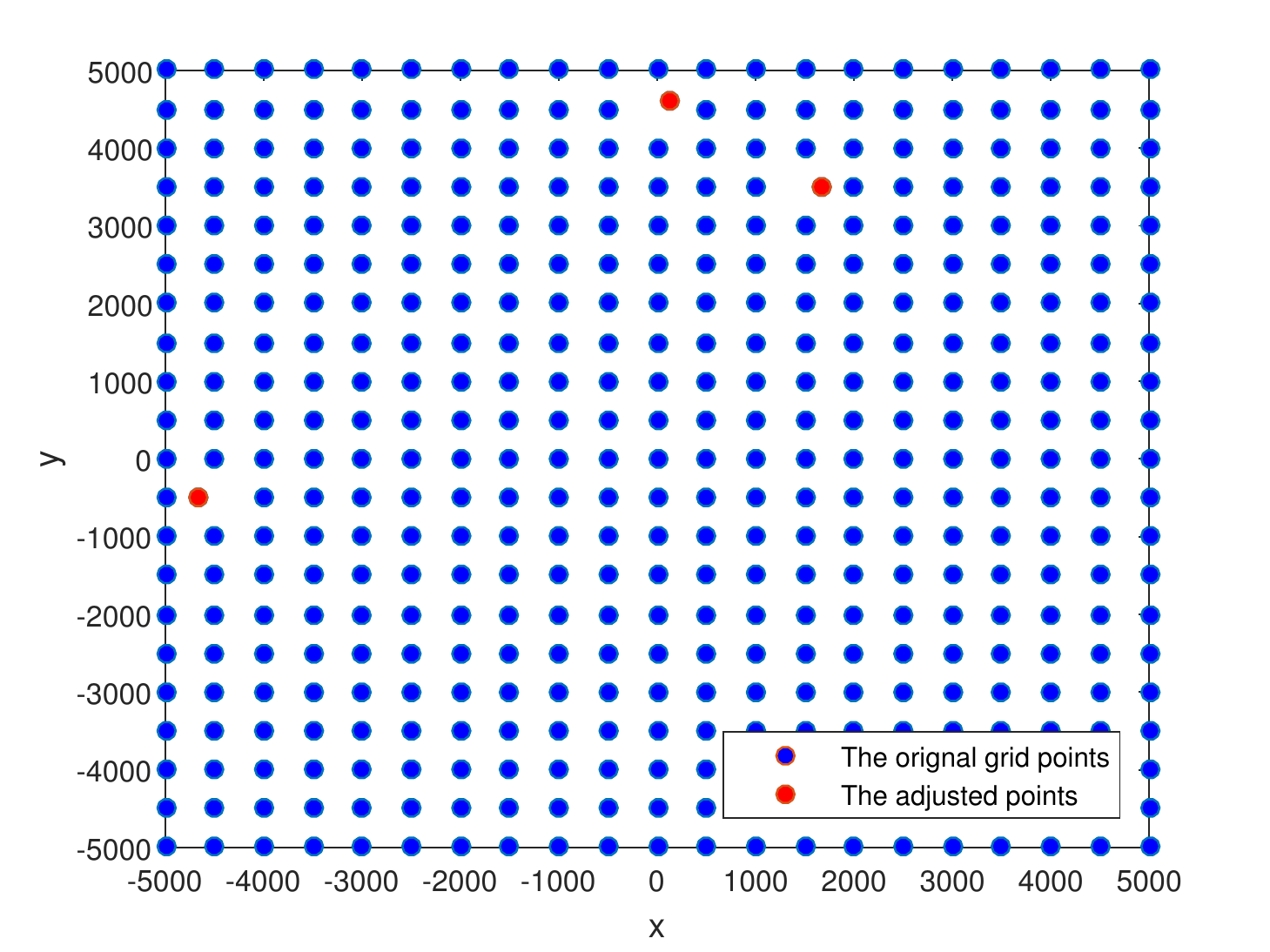}}
	\subfigure[]{\includegraphics[width=0.325\textwidth]{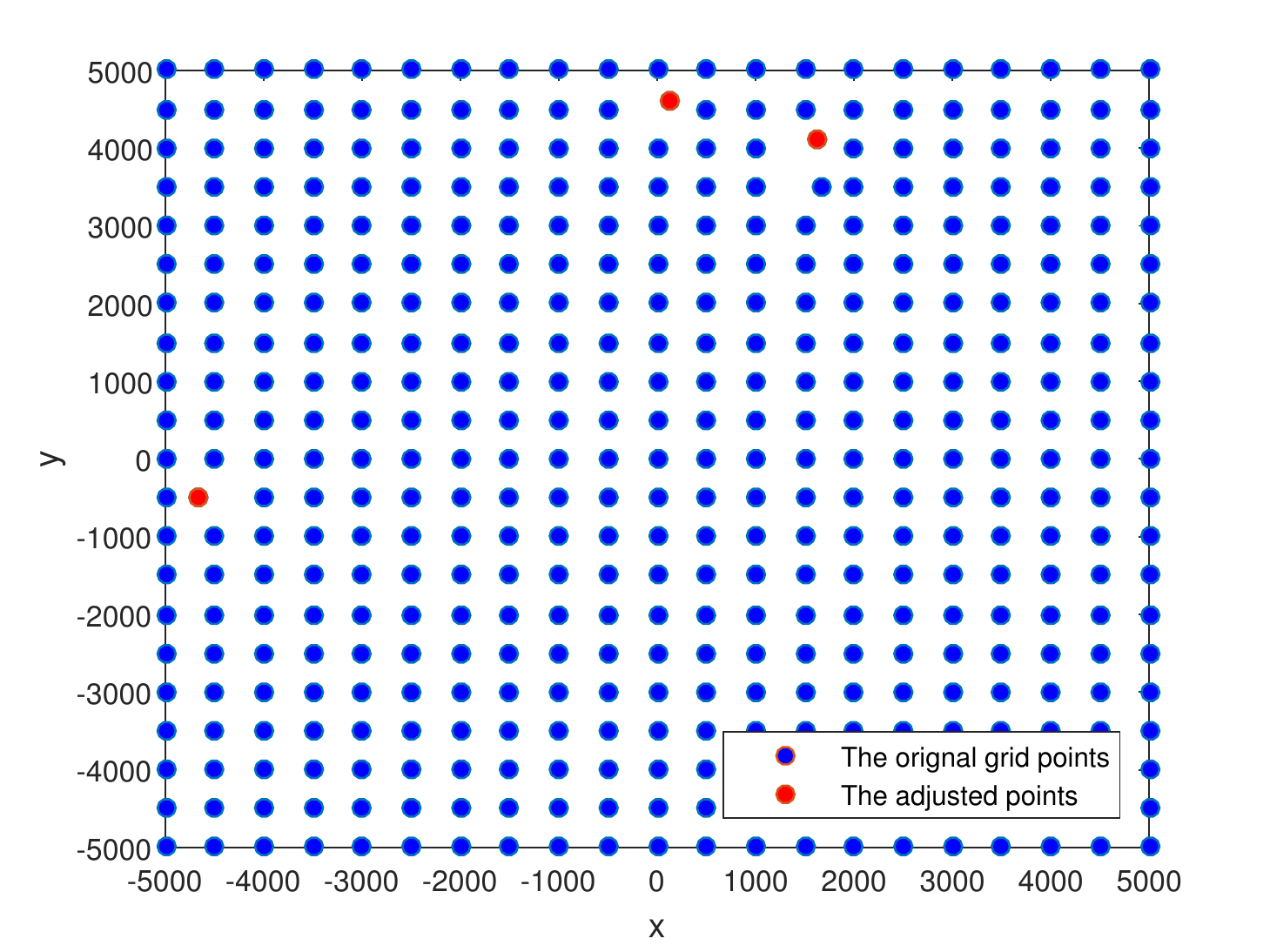}}
	\subfigure[]{\includegraphics[width=0.325\textwidth]{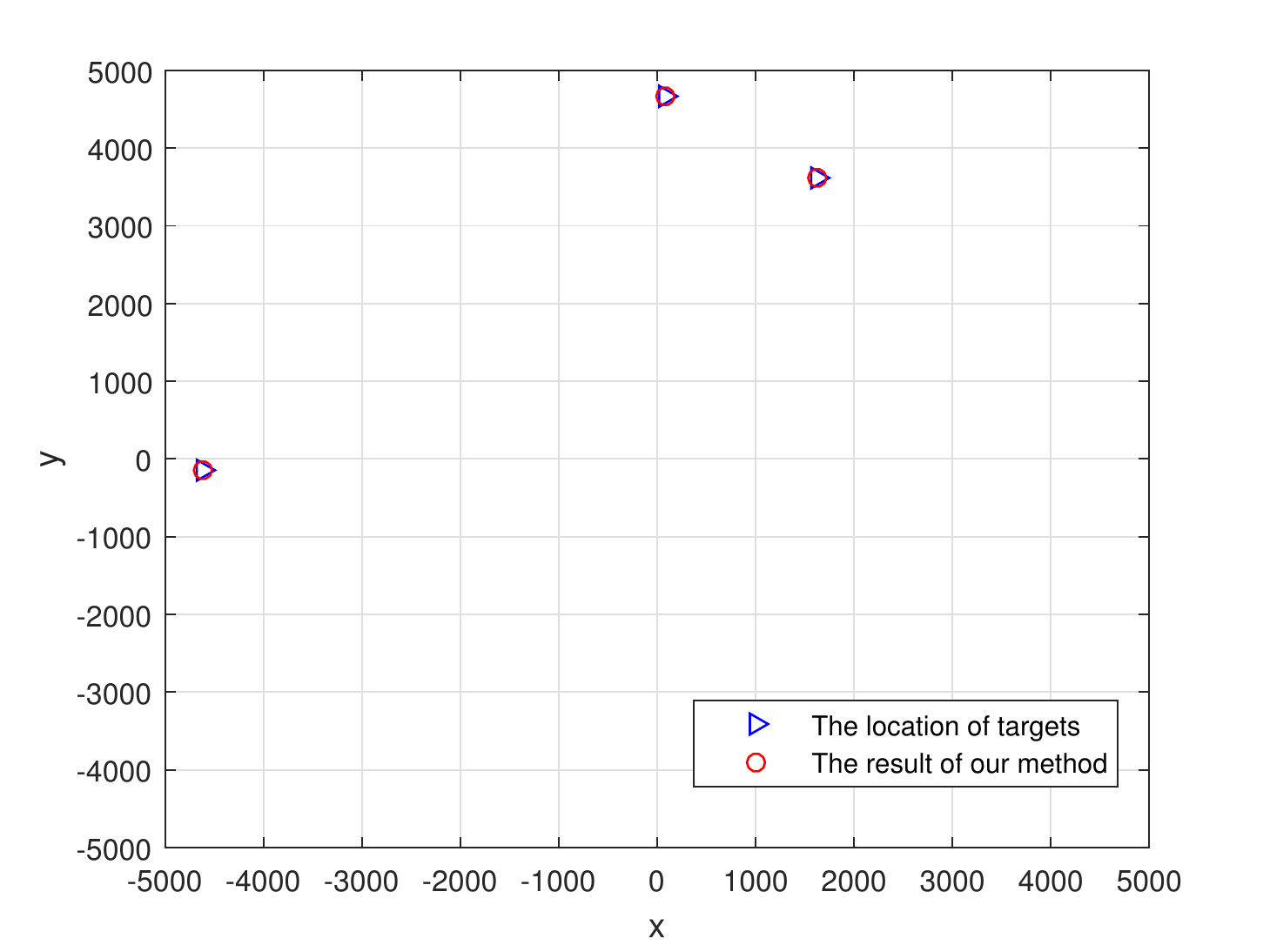}}
	\caption{The process of update method $\varphi _q$ . \label{fig5-5}}
\end{figure}

In Figure \ref{fig5-6}, Figure \ref{fig5-7} and Figure \ref{fig5-8}, the square zone is divided into 441 grid points, and we consider the case when the targets are located at the grid points or near the grid points. Compared to the classic sparse recovery algorithm, our method has a better success recovery ratio. In particular, we notice that 5 targets can be located by our method while 15 receivers are needed by classic TDOA methods.
\begin{figure}[!ht]
	\centering
	\subfigure[10 ns]{\includegraphics[width=0.32\textwidth]{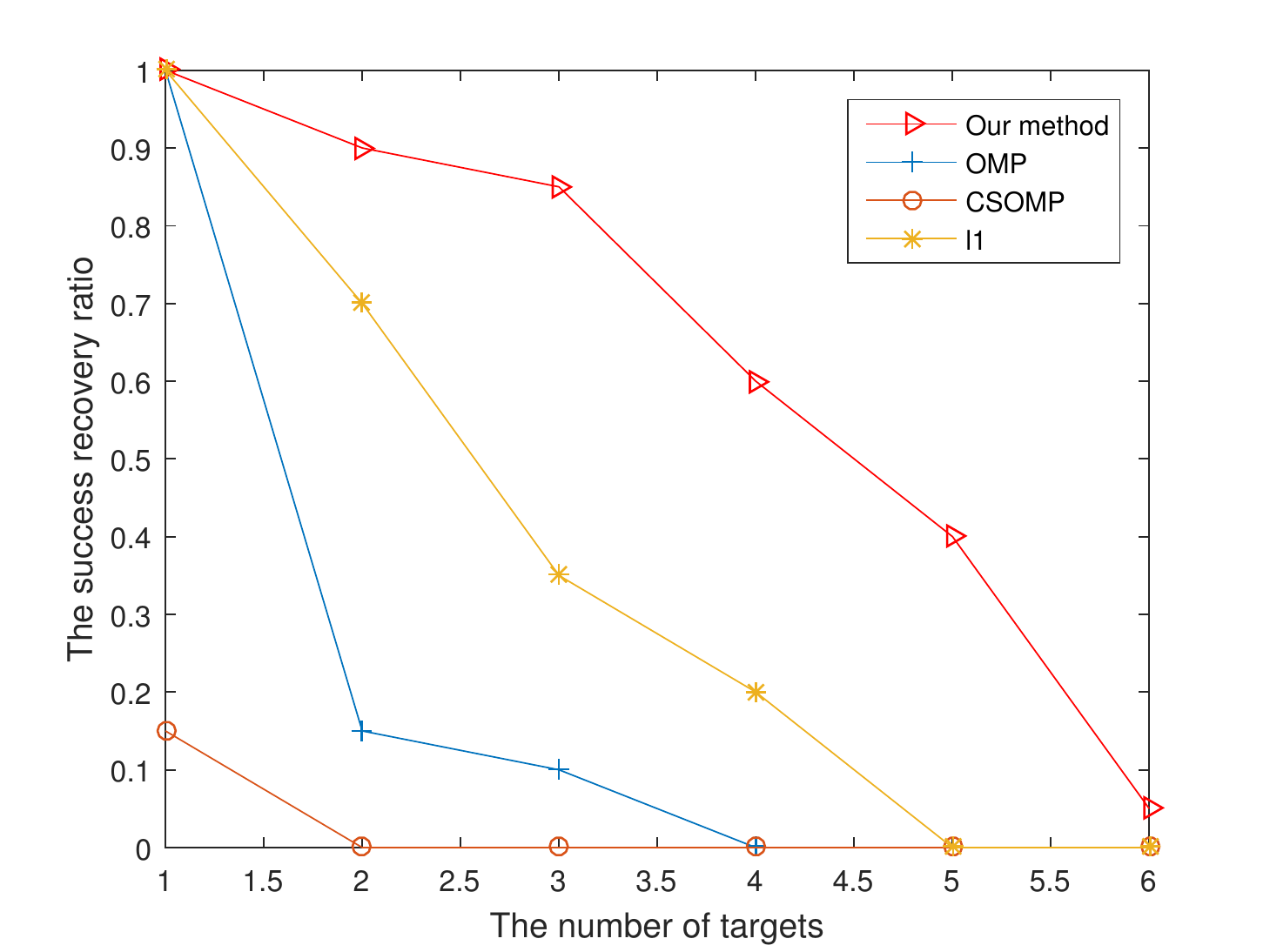}} \hfill
	\subfigure[5 ns]{\includegraphics[width=0.32\textwidth]{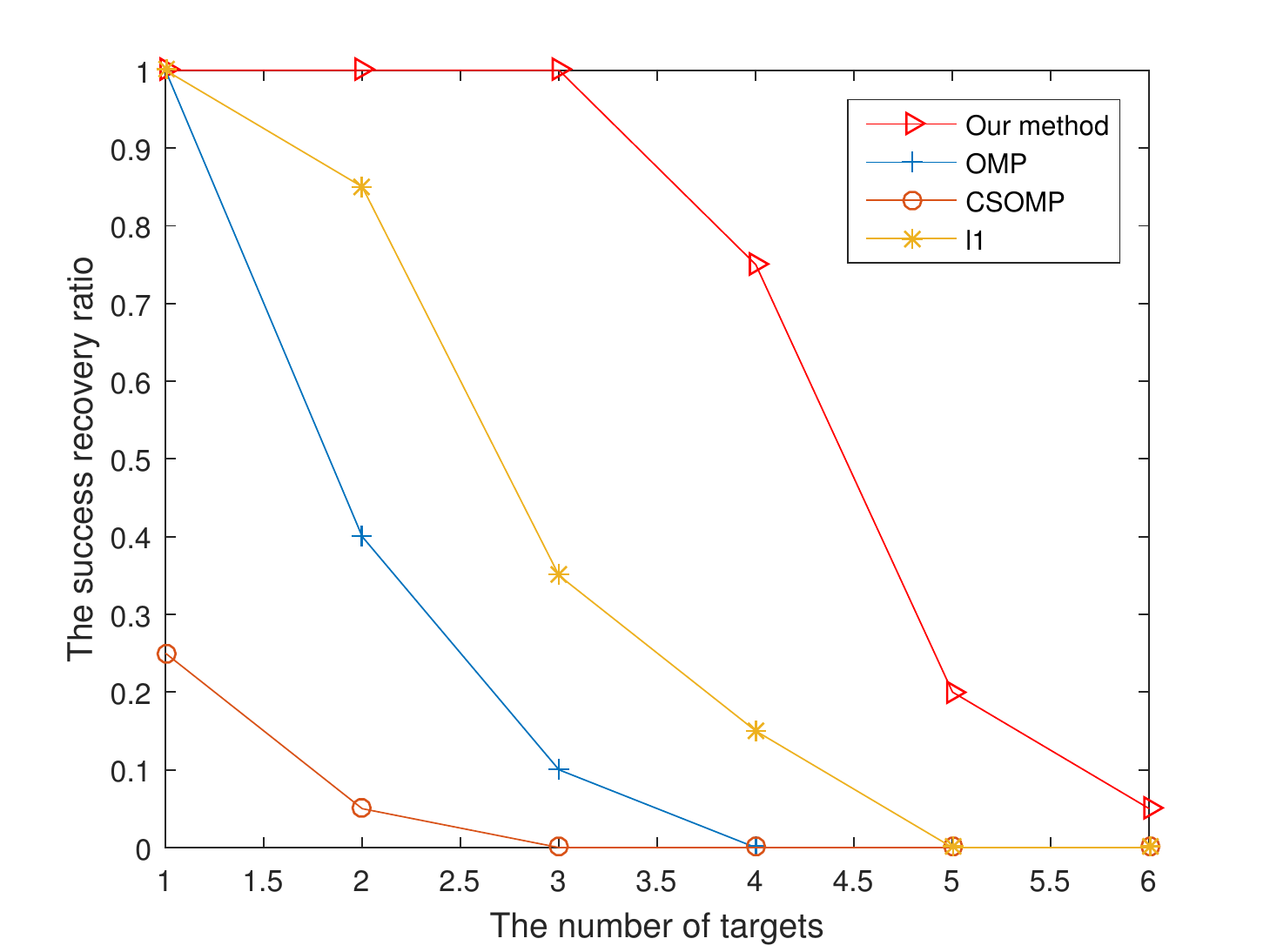}}
	\subfigure[1 ns]{\includegraphics[width=0.32\textwidth]{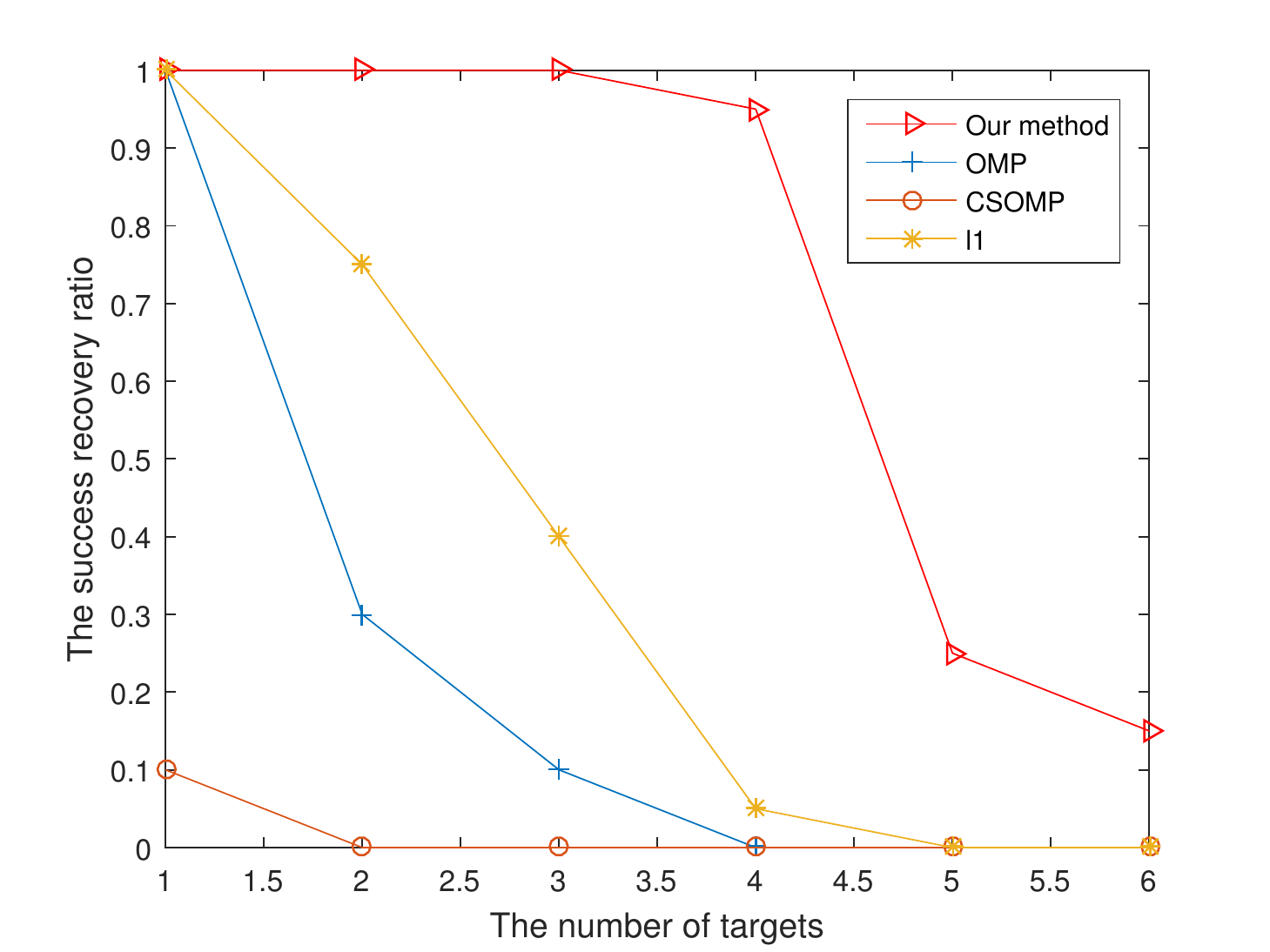}}
	\caption{With 6 sensors, the relationship between the number of tatgets and the sucess recovery ratio . \label{fig5-6}}
\end{figure}
\begin{figure}[!ht]
	\centering
	\subfigure[10 ns]{\includegraphics[width=0.32\textwidth]{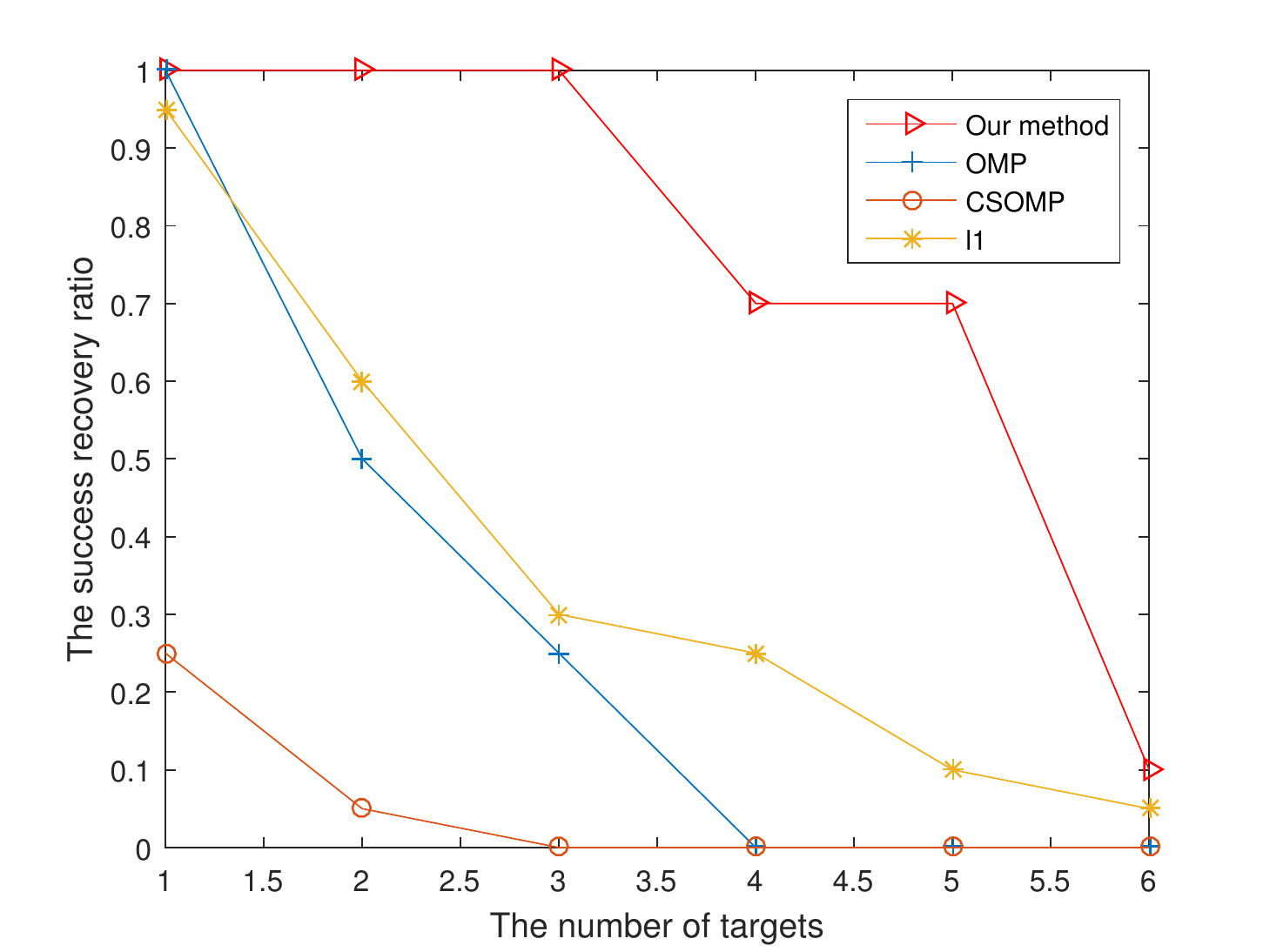}} \hfill
	\subfigure[5 ns]{\includegraphics[width=0.32\textwidth]{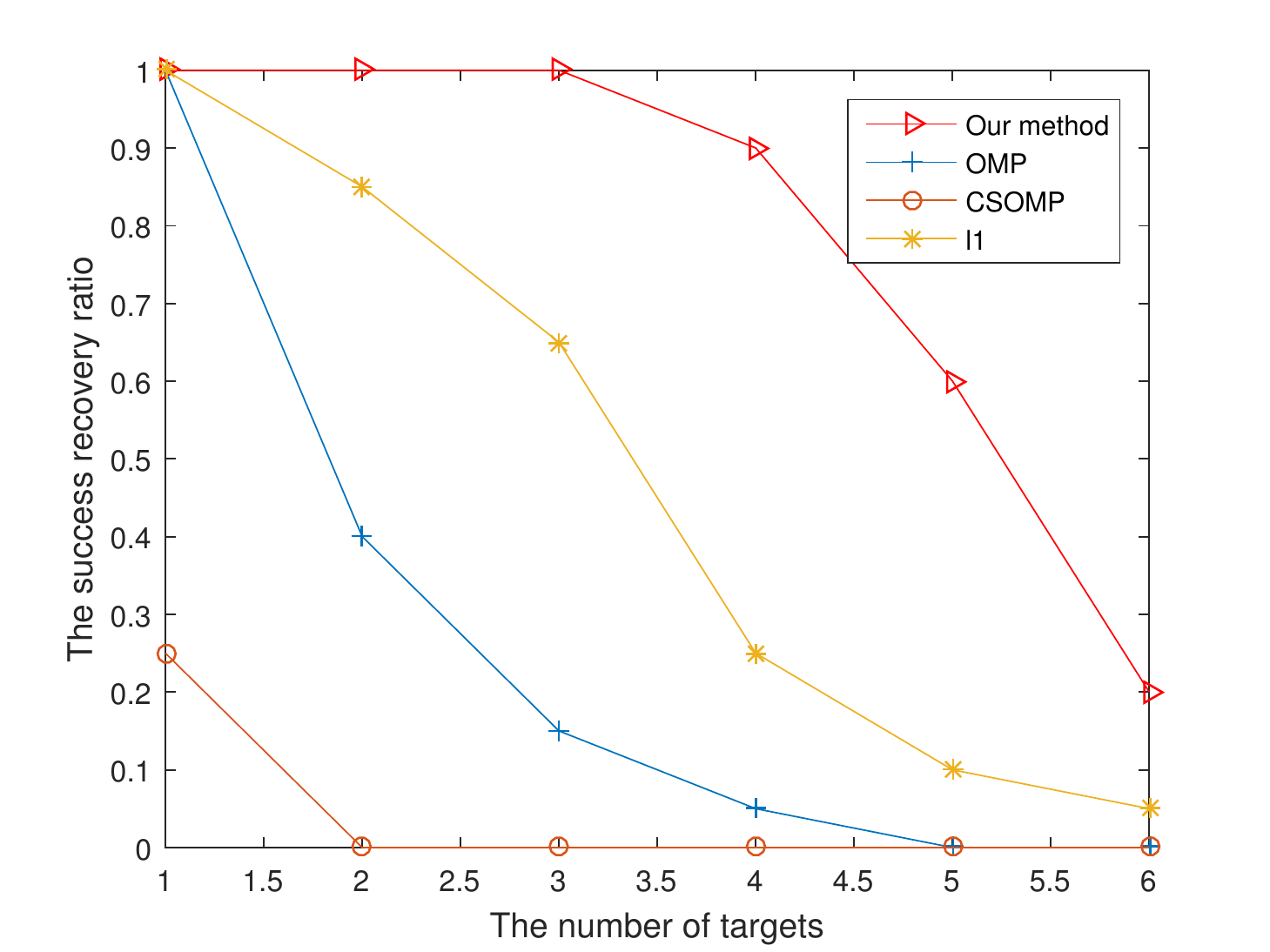}}
	\subfigure[1 ns]{\includegraphics[width=0.32\textwidth]{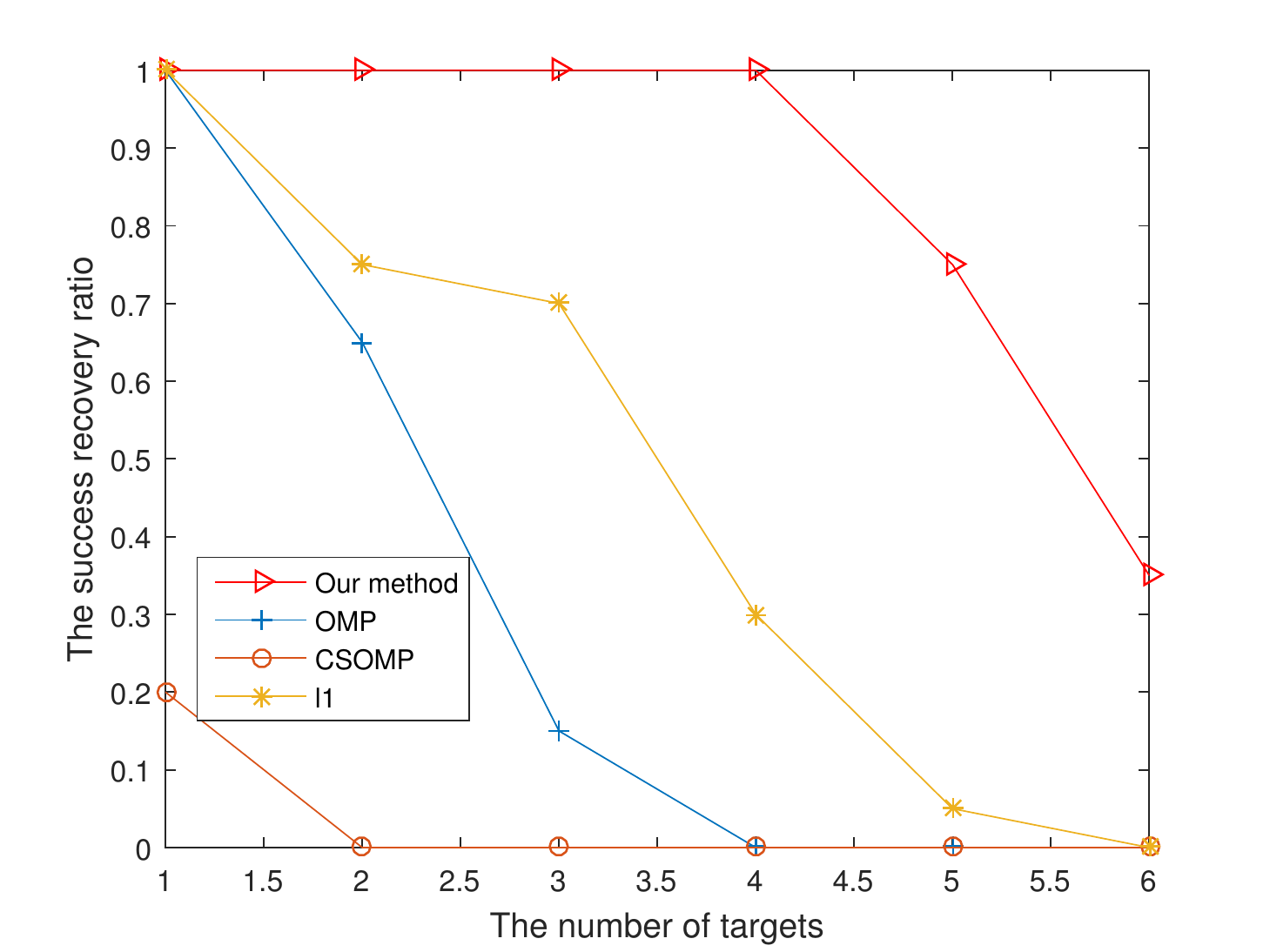}}
	\caption{With 7 sensors, the relationship between the number of tatgets and the sucess recovery ratio . \label{fig5-7}}
\end{figure}
\begin{figure}[!ht]
	\centering
	\subfigure[10 ns]{\includegraphics[width=0.32\textwidth]{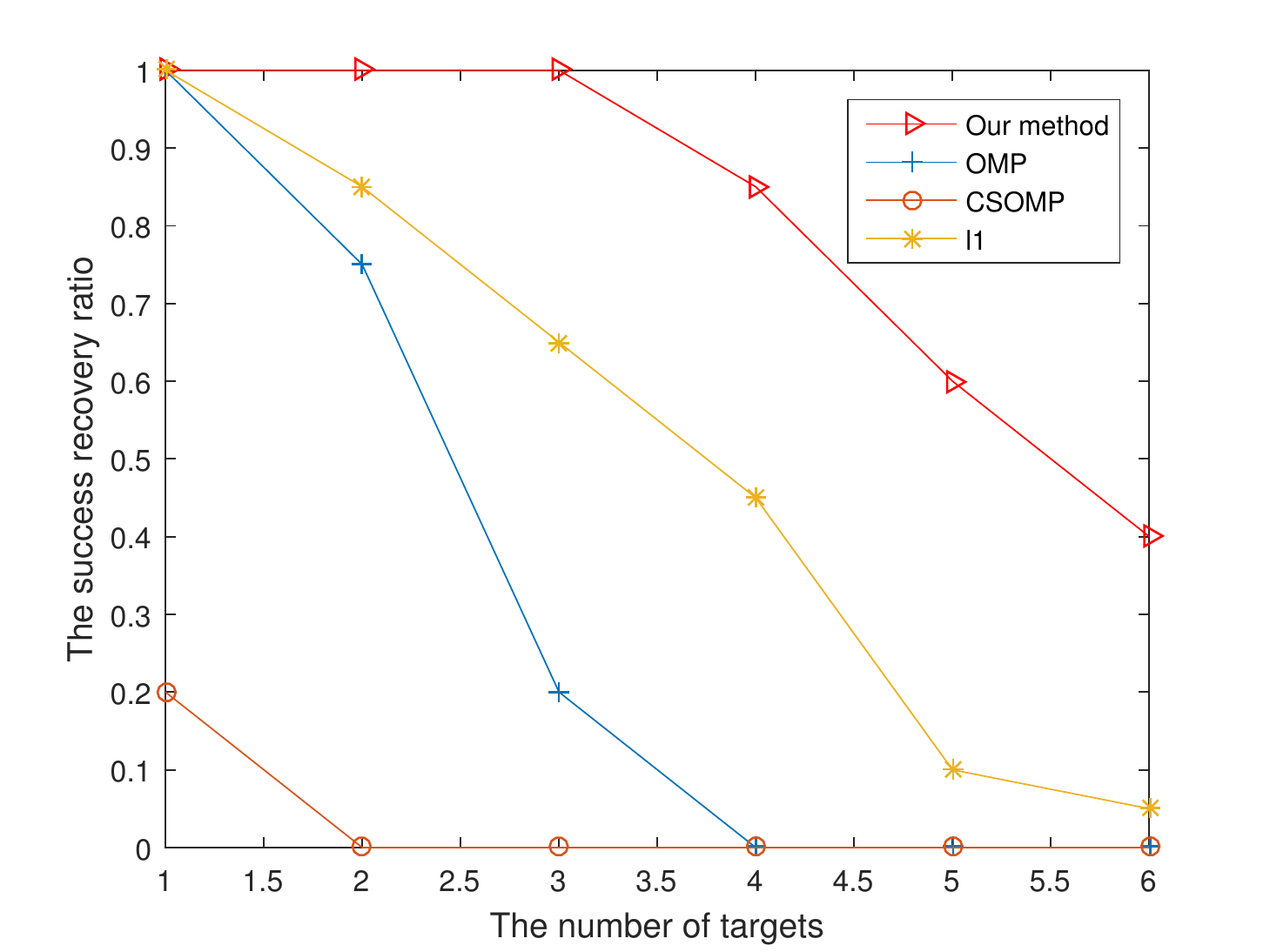}} \hfill
	\subfigure[5 ns]{\includegraphics[width=0.32\textwidth]{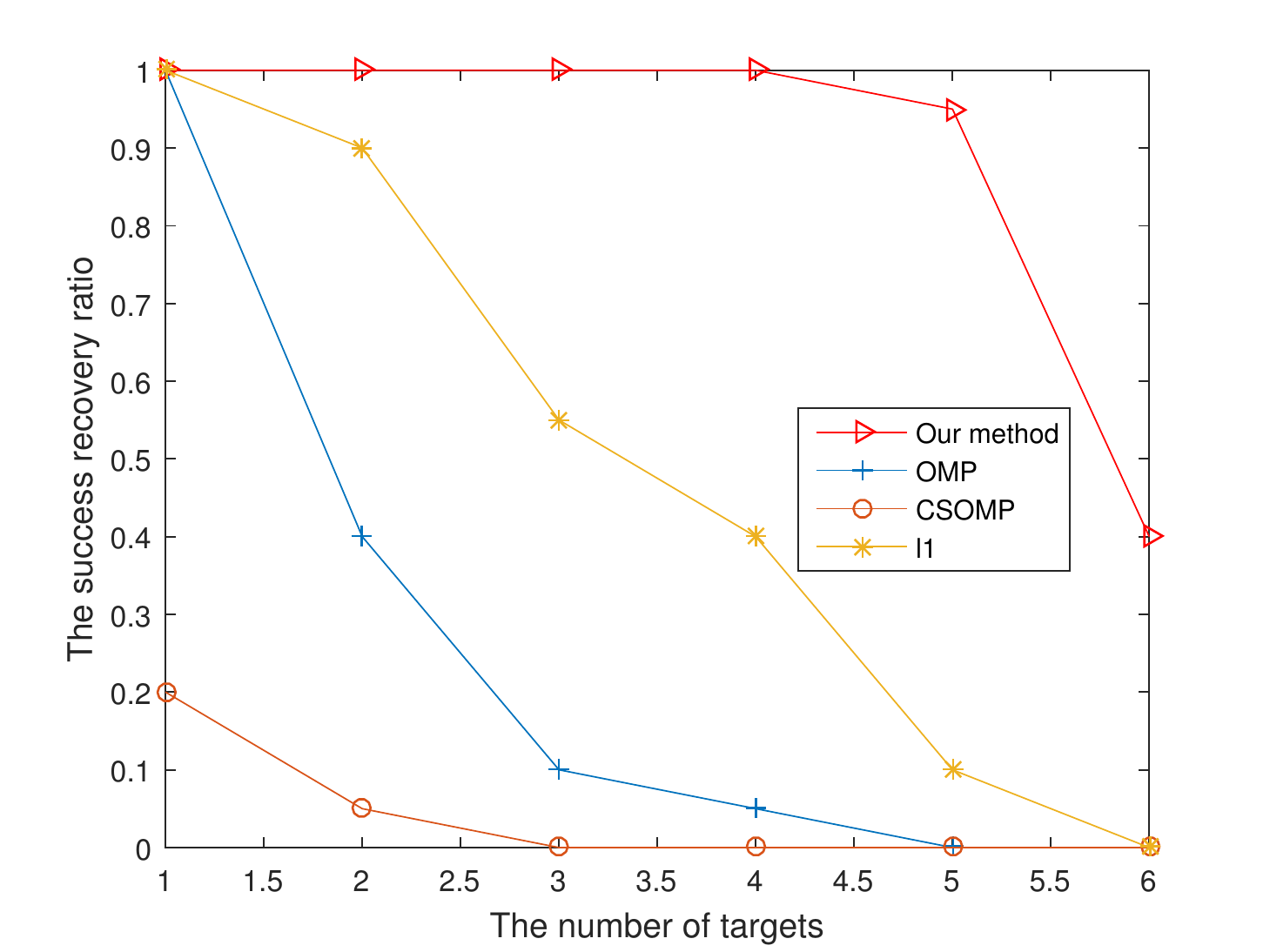}}
	\subfigure[1 ns]{\includegraphics[width=0.32\textwidth]{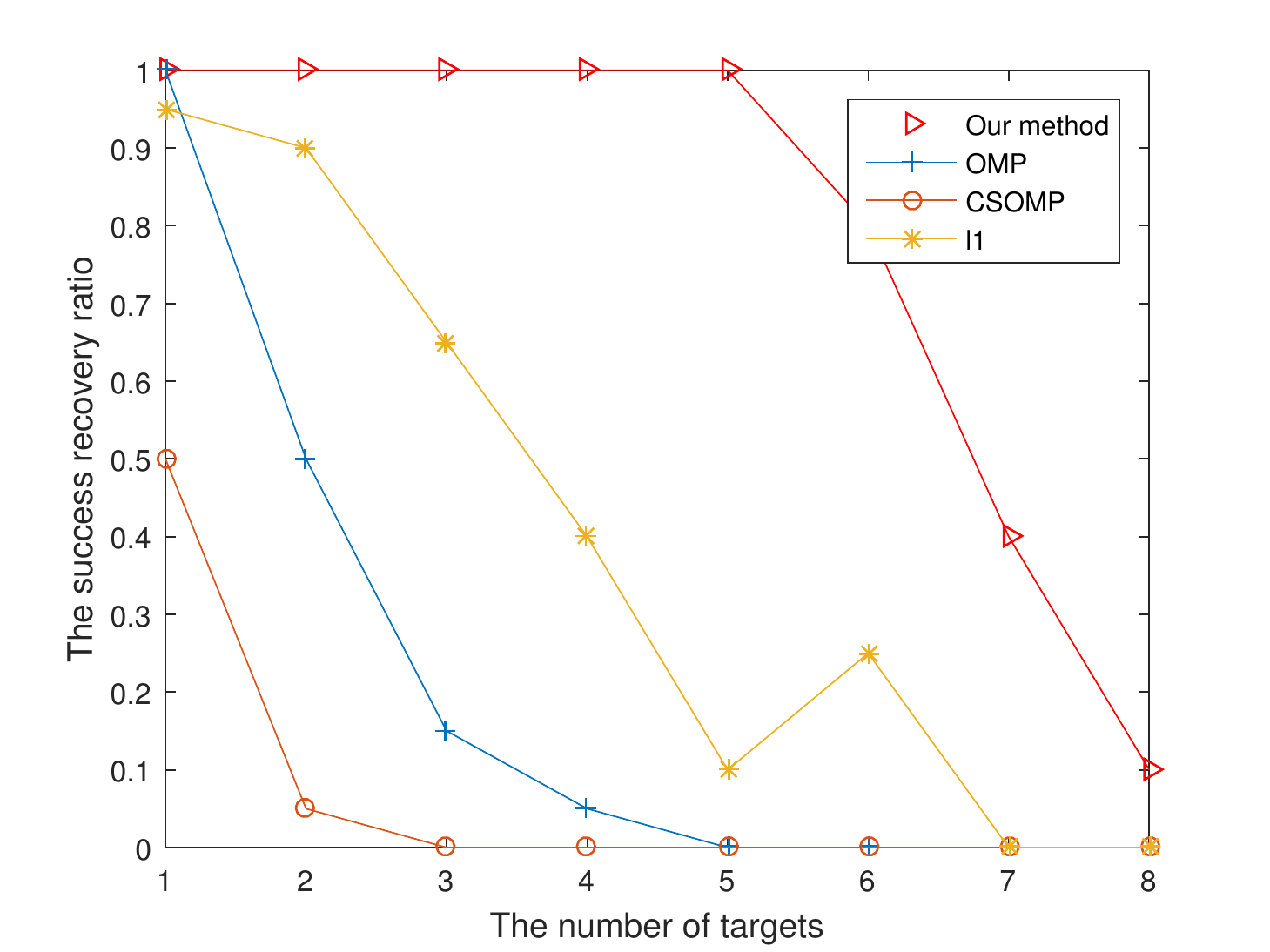}}
	\caption{With 8 sensors, the relationship between the number of tatgets and the sucess recovery ratio . \label{fig5-8}}
\end{figure}
In Figure \ref{fig5-9}, the targets are randomly picked and the results show us that our method can locate these targets properly and we show the average positioning error in Figure \ref{fig5-10}.
\begin{figure}[!ht]
	\centering
	\subfigure[Three targets ]{\includegraphics[width=0.49\textwidth]{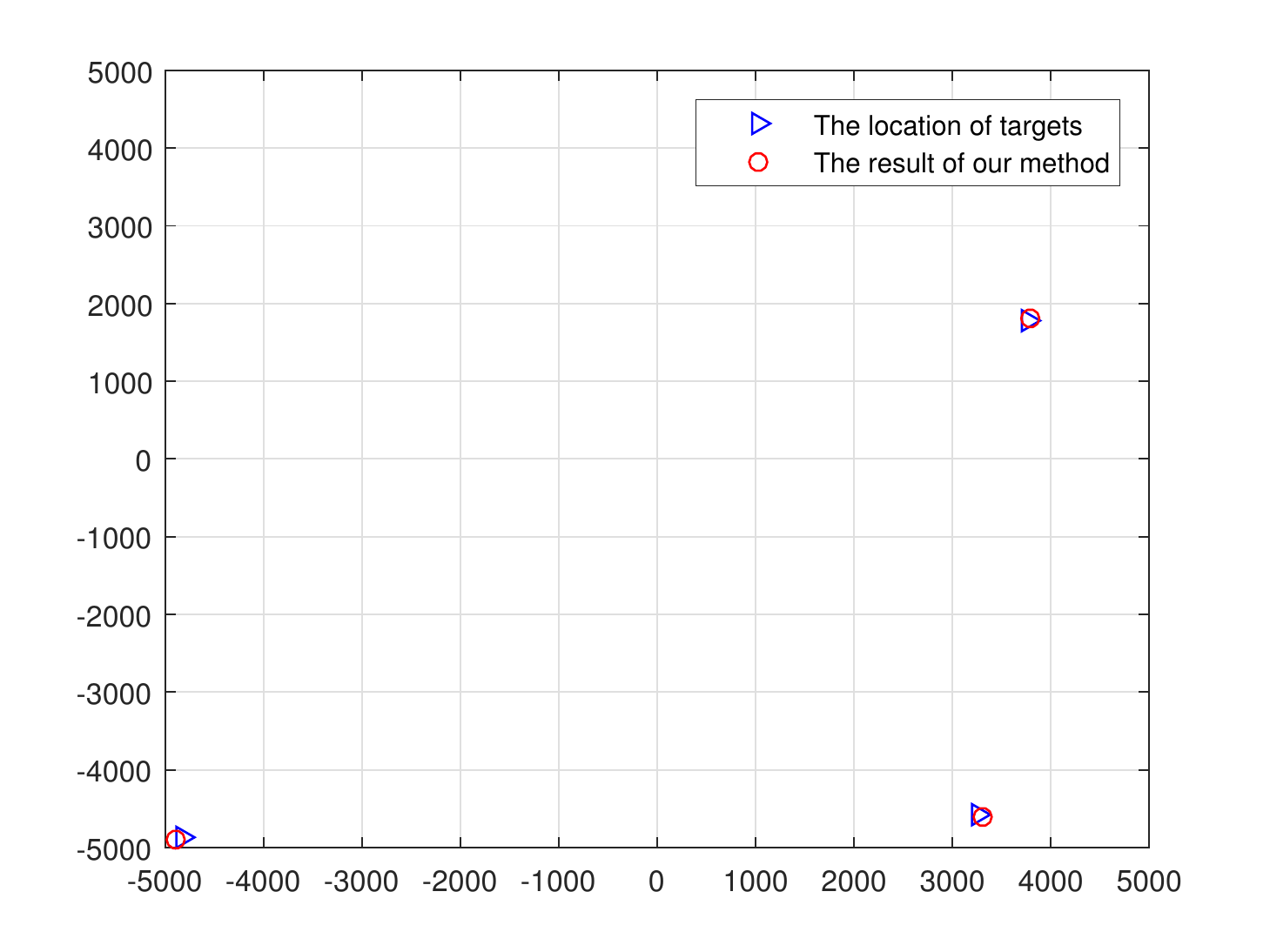}} \hfill
	\subfigure[Four targets]{\includegraphics[width=0.49\textwidth]{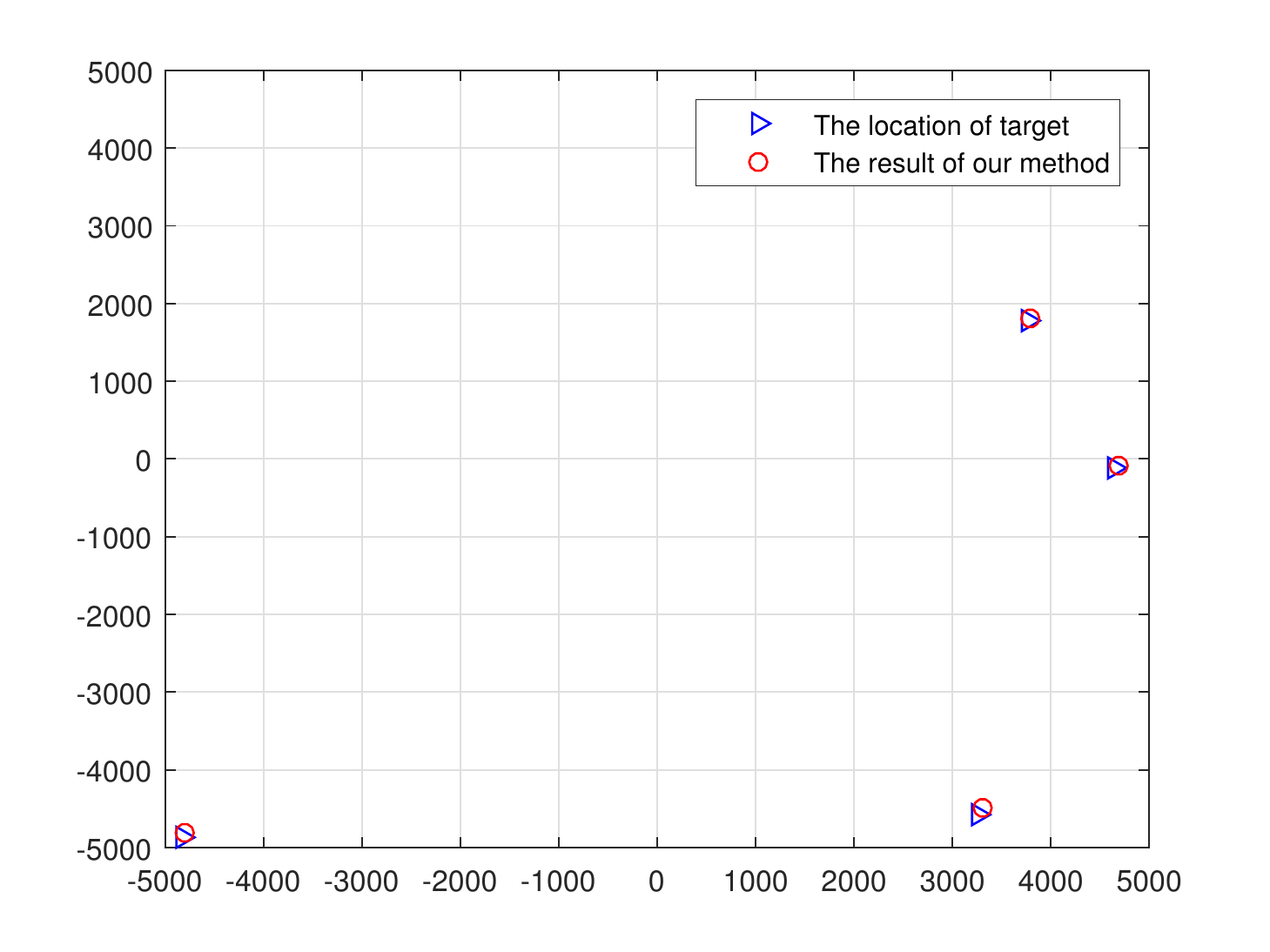}}
	\subfigure[Five targets]{\includegraphics[width=0.49\textwidth]{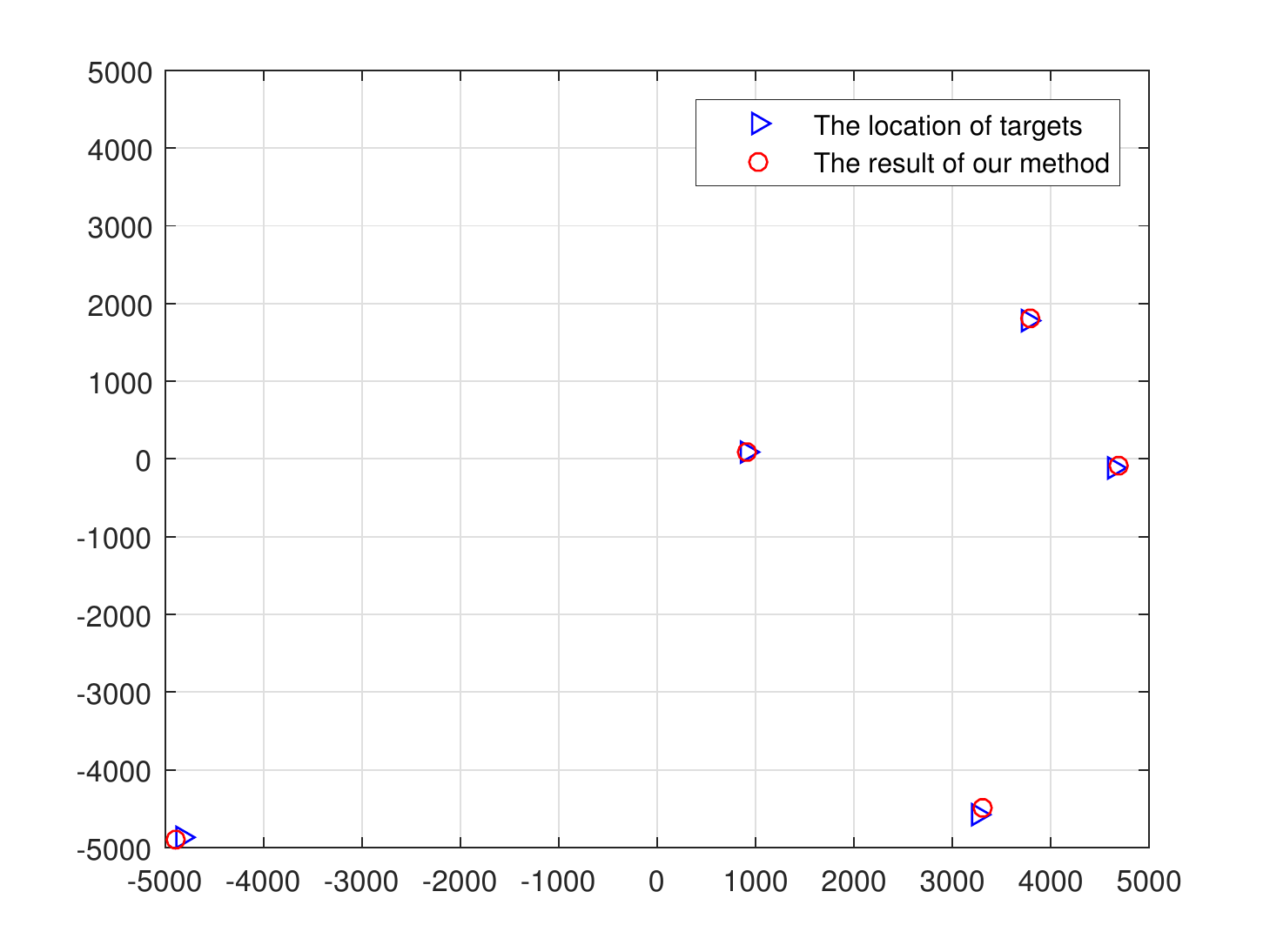}}
	\subfigure[Six targets]{\includegraphics[width=0.49\textwidth]{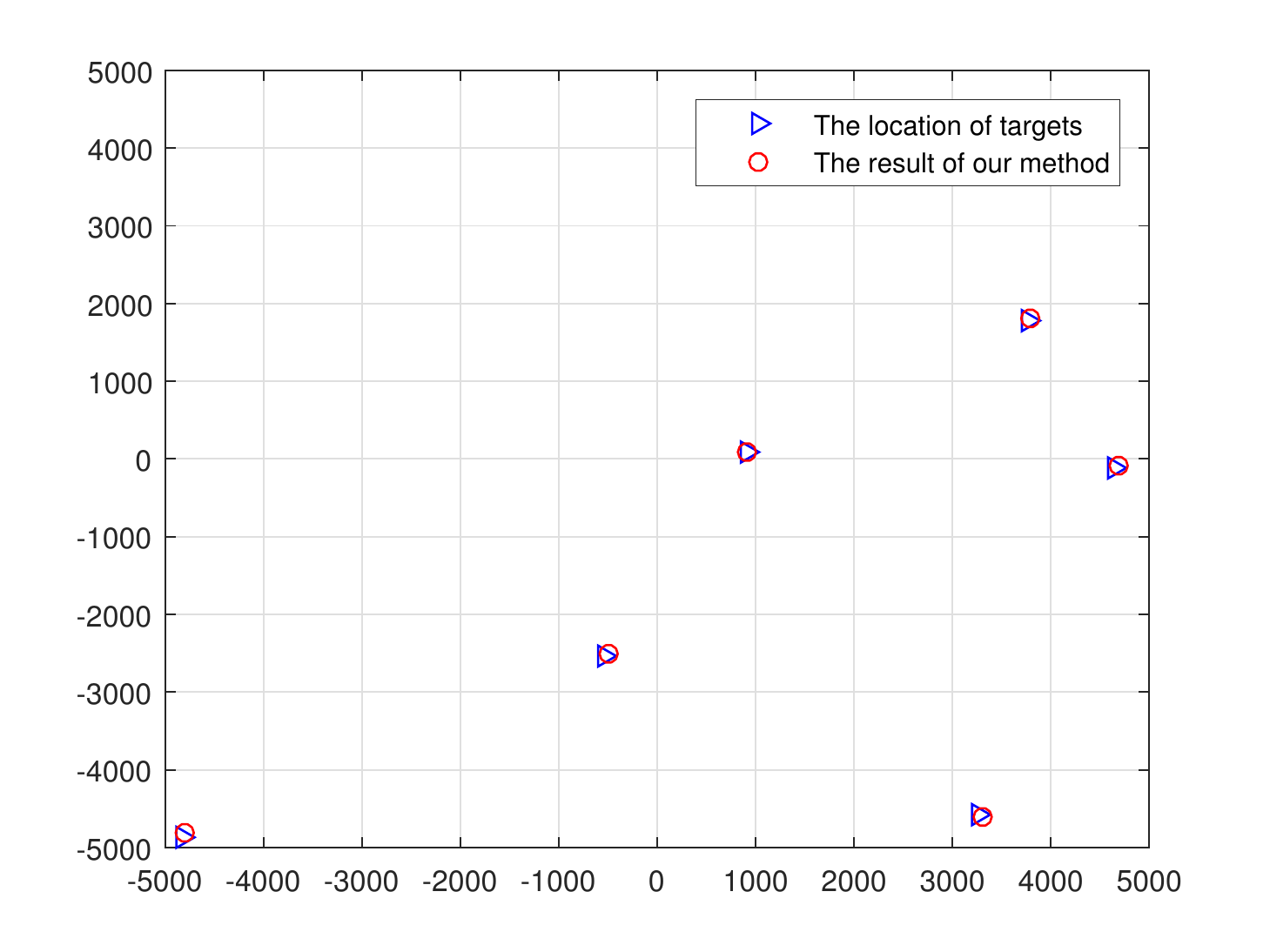}}
	\caption{The results of our method for muliti-source location problem. \label{fig5-9}}
\end{figure}
\begin{figure}[!ht]
	
	\centering
	\subfigure[The location of senors and targets]{\includegraphics[width=0.49\textwidth]{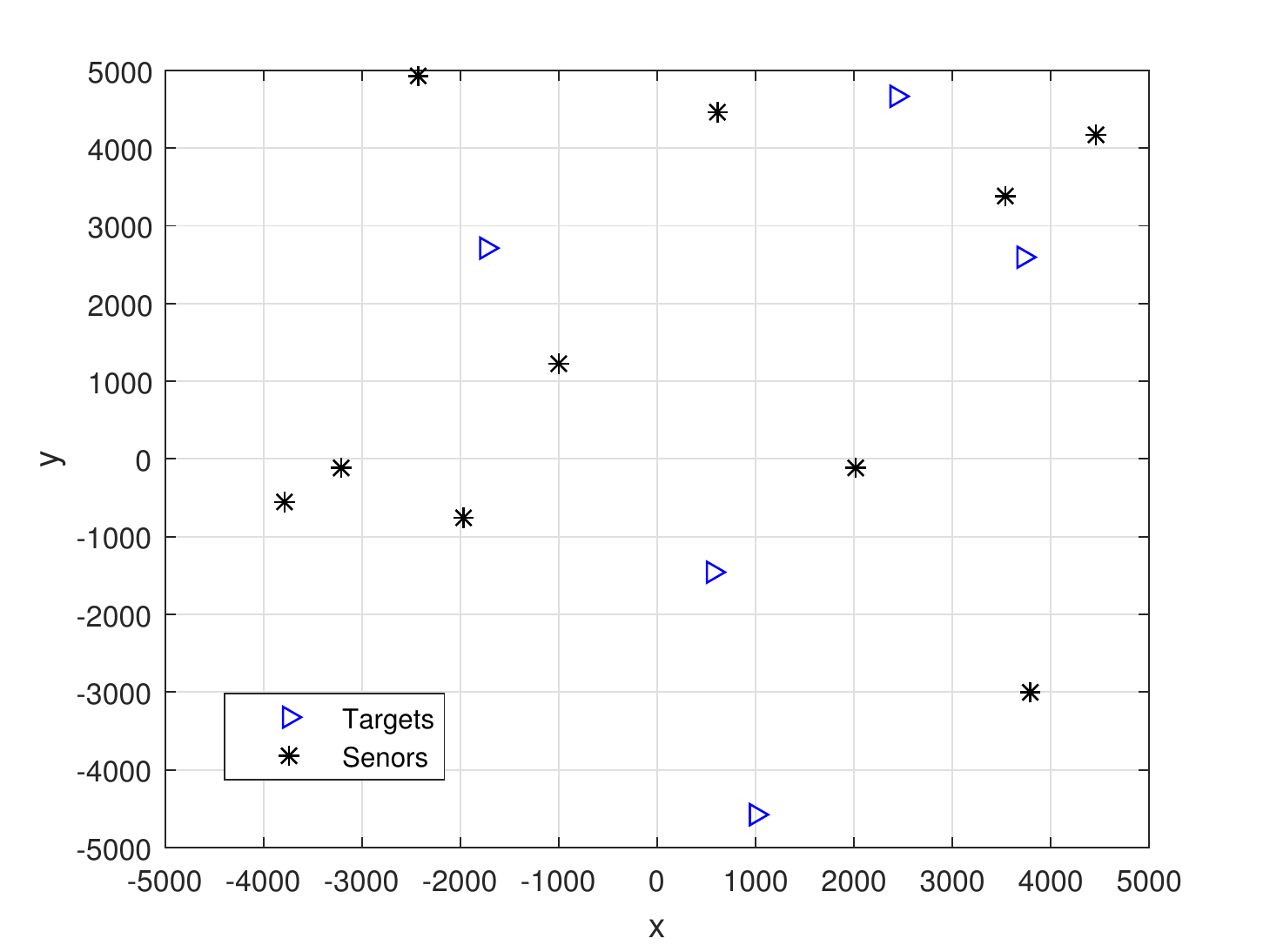}} \hfill
	\subfigure[The RMSE of different methods]{\includegraphics[width=0.49\textwidth]{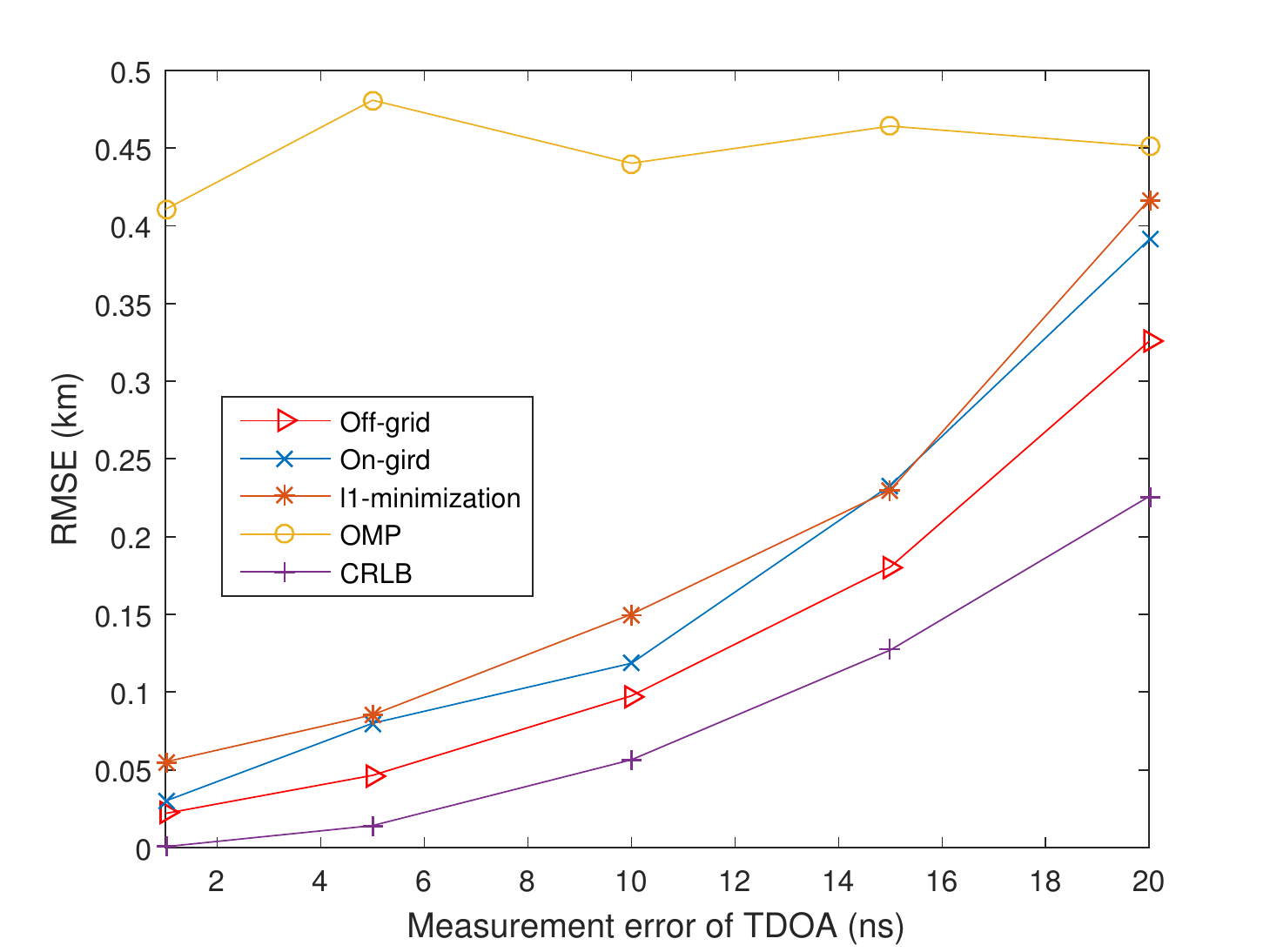}}
	
	\caption{The results of differnt method for muliti-source location problem . \label{fig5-10}}
\end{figure}

In Figure \ref{fig5-10}, we consider to locate 5 targets with 10 sensors. The performance of different algorithms is measured by the root mean square error (RMSE), which is defined as the average of error in dependent Monte Carlo trials. The result of our method designed for off-grid case is closer to CRLB than other algorithms.
\section{Conclusion}
To design a reasonable alternative model is the main method to solve $l_0$-minimization. In this paper, we consider the alternative function $h_{p,q}(x)$ since $h_{p,q}(0)=0$ and $\lim \limits_{p\rightarrow 0^+}\frac{h_{p,q}(x)}{log(1+p^{-1})}=1$. Furthermore, the equivalence relationship between these two models is presented and then we provide a necessary and sufficient condition for $l_{h_{p,q}}$-minimization. By a new concept named $H$-NSC, we prove that the recovery condition of $l_1$-minimization is more restrictive than that of $l_{h_{p,q}}$-minimization. Although $h_{p,q}(\cdot)$ is not a smooth function, we give an analysis expression of its local optimal solution and a fixed point algorithm. Finally, we use $l_{h_{p,q}}$-minimization to solve multiple source location problem. Compared to some classic algorithms, the result of our method is better than others. However, an analysis expression of $p^*(A,b,q)$ in Theorem 1 will improve the application of $l_{h_{p,q}}$-minimization. In conclusion, the authors hope that in publishing this paper, a brick will be thrown out and be replaced with a gem.  

\end{document}